\documentclass[11pt]{scrartcl}
\pdfoutput=1
\usepackage[utf8x]{inputenc}
\usepackage[T1]{fontenc}
\usepackage{lmodern}
\usepackage[english]{babel}

\usepackage[draft=false,kerning=true]{microtype}

\usepackage{amsmath,amssymb,amsthm,graphicx}
\usepackage{comment}
\usepackage{enumitem}
\usepackage{todonotes}
\usepackage{multirow}
\usepackage{mathtools}
\usepackage{tikz}
\usepackage{hyperref}

\newtheorem{theorem}{Theorem}[section]
\newtheorem{definition}[theorem]{Definition}
\newtheorem{notation}[theorem]{Notation}
\newtheorem{lemma}[theorem]{Lemma}
\newtheorem{corollary}[theorem]{Corollary}
\newtheorem{conjecture}[theorem]{Conjecture}

\newtheorem{proposition}[theorem]{Proposition}

\theoremstyle{remark}
\newtheorem{remarkx}[theorem]{Remark}
\newtheorem{examplex}[theorem]{Example}

\newenvironment{example}
    {\pushQED{\qed}\examplex}
    {\popQED\endexamplex}
\newenvironment{remark}
    {\pushQED{\qed}\remarkx}
    {\popQED\endremarkx}

\DeclarePairedDelimiter{\card}{\lvert}{\rvert}
\newcommand{\cG}{\mathcal{G}}
\newcommand{\cH}{\mathcal{H}}

\newcommand{\fviz}{f_{\mathrm{viz}}}
\newcommand{\Isu}{I_{\mathrm{su}}}
\newcommand{\Iviz}{I_{\mathrm{viz}}}
\newcommand{\K}{\mathbb{K}}
\newcommand{\Kbar}{\overline{\mathbb{K}}}
\newcommand{\cmN}{\mathit{cmN}}

\newcommand{\Q}{\mathbb{Q}}

\newcommand{\R}{\mathbb{R}}

\DeclarePairedDelimiter{\parentheses}{(}{)}
\DeclarePairedDelimiter{\set}{\{}{\}}
\newcommand{\variety}[1]{\mathcal{V}(#1)}
\newcommand{\Z}{\mathbb{Z}}
\newcommand{\SigSubset}[1]{\sigma_{g,#1}}
\newcommand{\SigPrimeSubset}[1]{\sigma_{g',#1}}
\newcommand{\tauij}[2]{\tau_{#1,#2}}
\newcommand{\SigIneq}[1]{\sigma^{\mathrm{I}}_{g,#1}}
\newcommand{\SigAll}[1]{\sigma^{\mathrm{A}}_{g,#1}}

\numberwithin{equation}{section}

\title{Towards a Computational Proof\\
     of Vizing's Conjecture \\
     using Semidefinite Programming \\ 
     and Sums-of-Squares}

\author{Elisabeth Gaar, Daniel Krenn,\\ Susan Margulies and Angelika Wiegele}
\date{}

\begin{document}

\maketitle

\begin{abstract}
  Vizing's conjecture (open since 1968) relates the product of the
  domination numbers of 
  two graphs to the domination number of their Cartesian product
  graph. In this paper, we formulate Vizing's conjecture
  as a Positivstellensatz existence question. In particular, we
  select classes of graphs according to their number of vertices and
  their domination number and 
  encode the conjecture as an ideal/polynomial pair such that the
  polynomial is non-negative on the variety associated with the ideal if and 
  only if the 
  conjecture is true
  for this graph class.
  Using semidefinite programming we  
  obtain numeric sum-of-squares certificates, which we then manage to 
  transform into symbolic certificates
  confirming non-negativity of our polynomials.
  Specifically,
  we obtain exact low-degree sparse sum-of-squares certificates for
  particular classes of graphs.

  The obtained certificates allow generalizations for larger graph classes. 
  Besides computational verification of these more general certificates, we also
  present theoretical proofs as well as conjectures and questions for
  further investigations.   
\end{abstract}

\bgroup%
\let\thefootnote\relax%
\footnotetext{An extended abstract containing the ideas of this
  optimization-based approach for tackling Vizing's conjecture
  appeared
  as~\cite{Gaar-Krenn-Margulies-Wiegele:2019:vizing-optimization-sos}.
  This article now also contains the full and complete proofs, new
  certificates only conjectured in the extended abstract, and further
  theoretical and computational results for particular cases.  This also
  lead to a restructuring of the whole article, and more examples and
  many more remarks explaining the implications of the results are
  provided.}%
\footnotetext{The authors gratefully acknowledge the support of
  Fulbright Austria (via a Visiting Professorship at
  Alpen-Adria-Universität Klagenfurt).
  This project has received funding from
  the European Union's Horizon~2020 research and innovation programme
  under the Marie Sk\l{}odowska-Curie grant agreement No~764759,
  the Austrian Science Fund (FWF): I\,3199-N31 and
  the Austrian Science Fund (FWF): P\,28466-N35.
}%
\egroup

\section{Introduction}

Sum-of-squares and its relationship to semidefinite
programming is a cutting-edge tool at the forefront of polynomial
optimization \cite{sos-instructions}. Activity in this area has
exploded over the past two decades to span areas as diverse as real
and convex algebraic geometry \cite{sos_geom}, control theory
\cite{sos_control}, proof complexity \cite{sos_proof}, theoretical
computer science \cite{sos_approx} and even quantum computation
\cite{sos_quantum}.  Systems of polynomial equations and other
non-linear models are similarly widely known for their compact and
elegant representations of combinatorial problems. Prior work on
polynomial encodings includes colorings \cite{alontarsi,hillarwindfeldt},
stable sets
\cite{susan_cpc,lovasz1}, matchings \cite{fischer},
and flows \cite{Onn}. In this project, we combine the
modeling strength of systems of polynomial equations with the
computational power of semidefinite programming and devise an
optimization-based framework for a computational proof of an old, open
problem in graph theory, namely Vizing's conjecture.

Vizing's conjecture was first proposed in 1968, and relates the sizes
of minimum dominating sets in graphs $G$ and $H$ to the size of a
minimum dominating set in the Cartesian product graph $G \Box
H$; a precise formulation follows as Conjecture~\ref{conjecture:vizing}.
Prior algebraic work on this conjecture \cite{susan_viz1}
expressed the problem as the union of a certain set of varieties and
thus the intersection of a certain set of ideals. However, algebraic
computational results have remained largely untouched. In this
project, we present an algebraic model of Vizing's conjecture that
equates the validity of the conjecture to the existence of a
Positivstellensatz, or a sum-of-squares certificate of non-negativity
modulo a carefully constructed ideal.

By exploiting the relationship between the Positivstellensatz and
semidefinite programming, we are able to produce sum-of-squares
certificates for certain classes of graphs where Vizing's conjecture
holds. Thus, not only are we demonstrating an optimization-based
approach towards a computational proof of Vizing's conjecture, but we
are presenting actual minimum degree non-negativity certificates that
are algebraic proofs of instances of this combinatorial problem.
Although the underlying graphs do not further what is known
about Vizing's conjecture at this time (indeed the
combinatorics of the underlying graphs is fairly easy), such a 
construction of ``combinatorial'' Positivstellens\"atze
is successfully executed for the first time here.
The construction process is an
elegant combination of computation, guesswork, computer algebra and 
proofs relying on clever definitions of certain polynomials as well as 
tricky manipulations.

Our paper is structured as follows. In Section~\ref{sec_back},
we present the necessary background and definitions from graph
theory and commutative algebra. In Section~\ref{sec_poly_I}, we
begin the heart of the paper: we describe the ideal/polynomial pair
that models Vizing's conjecture as a sum-of-squares problem. 
This pair is parametrized by the sizes~$n_\cG$ and $n_\cH$ of the
graphs~$G$ and $H$ respectively, and on the sizes $k_\cG$ and $k_\cH$
of a minimum dominating set in these graphs.
In
Section~\ref{sec:methodology}, we describe our precise process for finding the
sum-of-squares certificates along with an example.
In Sections~\ref{sec:proof_of_kG=nG_and_kH=nH-1}
and~\ref{sec:certificates_kGIsnGMinus1_kHEqnHMinus1}, we present our
computational results and the Positivstellens\"atze, i.e., the theorems that
arise for various generalizations.
In particular, in Section~\ref{sec:proof_of_kG=nG_and_kH=nH-1}, we
introduce certain symmetric polynomials that not only allow for a
compact notation, but also are vital in proving correctness of the
certificates. With the help of the developed calculus, we investigate the graph
classes where $k_\cG = n_\cG$ and $k_\cH = n_\cH-d$ and present 
certificates for $d\in \{0,\dots, 4\}$ (all other parameters
arbitrary). We provide formal proofs for $d\le 2$ and computational
proofs using SageMath~\cite{sagemath} for $d \le 4$.
Moreover, for fixed integer~$d$,
we explain an algorithm for computing a certificate or proving
that there is none of the conjectured form.
Then, in Section~\ref{sec:certificates_kGIsnGMinus1_kHEqnHMinus1}, the
case $k_\cG = n_\cG -1$ and $k_\cH = n_\cH -1$ is considered. For this
class, we obtain certificates for $n_\cH \in \{2,3\}$ ($n_\cG$
arbitrary) and prove their correctness.   
Finally, in Section~\ref{sec_conc}, we summarize our project,
state some concluding remarks and
present our ideas for future work.
For the sake of completeness, an appendix provides certificates (along
with proofs) that arose during the application of our method but were
dismissed after we obtained certificates with simpler forms.

The code accompanying this article can be found at
\url{https://gitlab.com/dakrenn/vizing-sdp-sos}.%
\footnote{The code at \url{https://gitlab.com/dakrenn/vizing-sdp-sos}
  is meant to be used with the
  open source mathematics software SageMath~\cite{sagemath} and
  the solver MOSEK~\cite{mosek} within MATLAB.}

\section{Backgrounds and Definitions}\label{sec_back}
In this section, we recall all necessary definitions from graph
theory, polynomial ideals and commutative algebra.

\subsection{Definitions from Graph Theory}

Given a graph $G$ with vertex set~$V(G)$,
a set $D \subseteq V(G)$ is a \emph{dominating set in $G$}
if for each $v \in V(G)\setminus D$, there is a $u \in D$ such that $v$ is
adjacent to $u$ (i.e., there is an edge between $u$ and $v$) in $G$. 
A dominating set is called \emph{minimum} if there is no 
dominating set of smaller size (i.e., cardinality).
The \emph{domination number of $G$}, denoted by
$\gamma(G)$, is the size of a minimum
dominating set in $G$. The
decision problem of determining whether a given graph has a dominating
set of size $k$ is NP-complete \cite{garey_and_johnson}.

Given graphs $G$ and $H$ with edge sets~$E(G)$ and $E(H)$ respectively,
the Cartesian product graph $G \Box H$ has vertex set%
\footnote{It will be convenient to use the short notation $gh = (g,h)$
  for an element of the vertex set~$V(G) \times V(H)$
  of the Cartesian product graph $G \Box H$.}
$V(G) \times V(H)$ and edge set
\begin{align*}
  E(G \Box H) = \big\{\set{gh, g'h'} \colon &\text{$g = g'$ and $\set{h,h'} \in E(H)$, or}\\
  &\text{$h = h'$ and $\set{g,g'} \in E(G)$}\big\}.
\end{align*}

In 1968, Vadim G.~Vizing conjectured the following beautiful relationship
between domination numbers and Cartesian product graphs:
\begin{conjecture}[Vizing \cite{vizing}, 1968]\label{conjecture:vizing}
Given graphs $G$ and $H$, then the inequality
\begin{equation*}
  \gamma(G)\,\gamma(H) \leq \gamma(G \Box H)
\end{equation*}
holds.
\end{conjecture}

\tikzset{
  vertex/.style = {circle, fill=gray!20, draw=black, minimum size=8pt, inner sep=0pt},
  edge/.style = {very thick},
  dominating/.style = {fill=gray},
}
\begin{example}\label{ex_prod_graph}
  In this example, we demonstrate the Cartesian product graph of two
  $C_4$ cycle graphs:
    
\begin{center}
  \begin{tikzpicture}[scale=0.7]

    \foreach \y/\sty in {0/, 1/dominating, 2/, 3/dominating}
    \node[vertex, \sty] (G-\y) at (-1.5, \y) {};
    
    \foreach \x/\sty in {0/dominating, 1/, 2/dominating, 3/}
    \node[vertex, \sty] (H-\x) at (\x, 4.5) {};
    
    \foreach \x/\y/\sty in {0/0/, 1/0/, 2/0/dominating, 3/0/,
                            0/1/dominating, 1/1/, 2/1/, 3/1/,
                            0/2/, 1/2/, 2/2/dominating, 3/2/,
                            0/3/dominating, 1/3/, 2/3/, 3/3/}
    \node[vertex, \sty] (GH-\x-\y) at (\x, \y) {};

    \foreach \xa/\xb in {0/1, 1/2, 2/3}
    {
      \foreach \y in {0, 1, 2, 3}
      {
        \draw [edge] (GH-\xa-\y) -- (GH-\xb-\y);
        \draw [edge] (GH-3-\y) to[out=155, in=25] (GH-0-\y);
      }
      \draw [edge] (H-\xa) -- (H-\xb);
    }
    \draw [edge] (H-3) to[out=155, in=25] (H-0);

    \foreach \ya/\yb in {0/1, 1/2, 2/3}
    {
      \foreach \x in {0, 1, 2, 3}
      {
        \draw [edge] (GH-\x-\ya) -- (GH-\x-\yb);
        \draw [edge] (GH-\x-3) to[out=-115, in=115] (GH-\x-0);
      }
      \draw [edge] (G-\ya) -- (G-\yb);
    }
    \draw [edge] (G-3) to[out=-115, in=115] (G-0);

    \node at (-1.5, 3.75) {$G=C_4$};
    \node at (-1.25, 4.5) {$H=C_4$};
    \node at (1.5, -0.75) {$G \Box H$};
  \end{tikzpicture}
\end{center}

In these graphs,~\tikz{\node[vertex, dominating] {};} represents a
vertex in a minimum dominating set, and Vizing's conjecture holds with equality:
$\gamma(G)\,\gamma(H) = 2\cdot 2 = 4 = \gamma(G \Box H)$. However,
observe that some copies of $G$ in $G\Box H$ do not contain any
vertices of the dominating set, i.e., they are dominated entirely
by vertices in other ``layers'' of the graph. This example highlights
the difficulty of Vizing's conjecture.
\end{example}

\subsection{Historical Notes}
Vizing's conjecture is an active area of research spanning over fifty
years. Early results have focused on proving the conjecture for
certain classes of graphs. For example, in 1979, Barcalkin and
German \cite{viz_bar_ger} proved that Vizing's conjecture holds for
graphs satisfying a certain ``partitioning condition'' on the vertex
set. The idea of a ``partitioning condition'' inspired work for the
next several decades, as Vizing's conjecture was shown to hold on
paths, trees, cycles, chordal graphs, graphs satisfying certain
coloring properties, and graphs with $\gamma(G) \leq 2$. These results
are clearly outlined in the 1998 survey paper by Hartnell and Rall
\cite{viz_dom_bk}.
In 2004,
Sun~\cite{viz_sun} showed that Vizing's conjecture holds on graphs
with $\gamma(G)\leq 3$.
There are also results proving a weaker version of the conjecture, for example,
the recent result of Zerbib~\cite{Zerbib2019}
showing that $\gamma(G)\,\gamma(H) + \max \{\gamma(G),\gamma(H)\} \leq 2\, \gamma(G \Box H)$.
The 2009 survey paper~\cite{viz_survey_2009}
summarizes the work from 1968 to 2008, contains new results,
new proofs of existing results, and comments about minimal
counter-examples.

\subsection{Definitions around Polynomial Ideals and Sum-of-Squares}
Our goal is to model Vizing's conjecture as a semidefinite program. 
In particular, we will create an ideal/polynomial pair such
that the polynomial is non-negative on the variety associated with the ideal
 if
and only if Vizing's conjecture is true. 

In this subsection, we present a brief introduction to polynomial
ideals, and the relationship between non-negativity and
sum-of-squares. This material is necessary for understanding our
polynomial ideal model of Vizing's conjecture. For a more thorough
introduction to this material see \cite{sos-instructions} and~\cite{coxetal}.

Throughout this section,
let $I$ be an ideal in a polynomial ring~$P=\K[z_1,\ldots,z_n]$
over a field $\K \subseteq \R$.
The \emph{variety of the ideal~$I$} is defined as
the set
\begin{equation*}
  \variety{I} = \{z^\ast\in\Kbar^n \colon \text{$f(z^\ast)=0$~for all $f \in I$}\}
\end{equation*}
with $\Kbar$ being the algebraic closure of~$\K$.
The variety $\variety{I}$ is called \emph{real}
if $\variety{I} \subseteq \R^n$.

We
say that the ideal~$I$ is \emph{radical}
if whenever $f^m \in I$ for some polynomial~$f\in P$ and integer
$m \geq 1$, then $f \in I$.
It should be mentioned that radical ideals and varieties are closely connected.%
\footnote{If the ideal~$I$ is radical, then $I = I(\variety{I})$
  where $I(\variety{I})$ is the ideal vanishing on $\variety{I}$.
  We do not need this statement in our paper explicitly---the spirit
  of the very same, however, is present.}

The concrete ideals that we are using later on are all radical. This is a
consequence of the following lemma.

\begin{lemma} (\cite[Section 3.7.B, pg. 246]{kreuzer})\label{lem_kreuzer_rad}
  Let $I$ be an ideal with finite variety~$\variety{I}$.
  If the ideal~$I$ contains a univariate
  square-free polynomial in each variable, then $I$ is radical.
\end{lemma}
The notion \emph{square-free} implies that when a polynomial is
decomposed into its unique factorization, there are no repeated
factors. For example, $(z_1^2 + z_2)(z_1^4 + 2z_3 + 3)$ is square-free,
but~$(z_1^2 + z_2)(z_1^4 + 2z_3 + 3)^3$ is not.
In particular, Lemma
\ref{lem_kreuzer_rad} implies that ideals containing
$z_i^2 - z_i = z_i(z_i - 1)$ in each variable (i.e., boolean
ideals) are radical.

In this work, we will make heavy use of one of the most prominent
theorems of algebra, namely Hilbert's Nullstellensatz.
\begin{theorem}[Hilbert's Nullstellensatz]
    \label{thm:HilbertsNullstellensatz}
    Let $\K$ be a field (not necessarily real, as assumed everywhere else),
    $P=\K[z_1, \dots, z_n]$ a
    polynomial ring, $I \subseteq P$ an ideal and $f \in P$.
    If $f(z^\ast)=0$ for all $z^\ast\in\variety{I}$, then there is a
    non-negative integer~$r$ with $f^r \in I$.
\end{theorem}

\begin{remark}\label{remark:HilbertsNullstellensatz}
    Our set-up implies the following:
    \begin{itemize}
        \item If the ideal~$I$ is radical, then
        $f(z^\ast)=0$ for all $z^\ast\in\variety{I}$ implies $f \in I$.
        \item If $I = \langle f_1,\ldots, f_q \rangle$ for some $f_1$,
        \dots, $f_q \in P$, then it suffices to check all $z^\ast$ that are 
        common
        zeros of $f_1$, \dots, $f_q$ (over the algebraic closure
        $\Kbar$) instead of all $z^\ast\in\variety{I}$.
    \end{itemize}
    Therefore, if both assumptions are satisfied,
    then $f(z^\ast)=0$ for all $z^\ast$ that
    are common zeros of $f_1$, \dots, $f_q$ (over the algebraic closure
    $\Kbar$) implies $f \in I$.
\end{remark}

We continue with our background by recalling the
necessary notation for sum-of-squares
for the ideal~$I$ of the polynomial ring~$P$.
As usual, we write $f \equiv h \mod I$ whenever $f = h + g$ for some $g \in I$.
\begin{definition}\label{definition:ellsos}
  Let $\ell$ be a non-negative integer. A polynomial
  $f \in P$ is called \emph{$\ell$-sum-of-squares modulo $I$} (or 
  \emph{$\ell$-sos modulo $I$}), if there exist polynomials $s_{1}$, \dots, $s_{t} 
  \in P$ with degrees
  $\deg{s_{i}} \leq \ell$ for all $i \in \set{1, \dots,t}$ and
  \begin{equation*}
    f \equiv \sum_{i=1}^t s_{i}^{2} \mod I.
  \end{equation*}
\end{definition}
In the context of real-valued polynomials as we have it,
algebraic identities like $f = \sum_{i=1}^t s_{i}^{2} + g$ for some $g\in I$,
are often referred to as
\emph{Positivstellensatz certificates of non-negativity}, and these
identities can be found via semidefinite programming, which is at the
heart of this project.
  We present a first example now and will describe precisely how these 
  certificates are obtained in
  Section~\ref{sec:methodology}.
\begin{example}\label{ex_sos}
  Let 
  $I = \big\langle z_1^2-z_1, z_2^2-z_2, z_1z_2 - 1\big\rangle$.
  In this case,
  \begin{align*}
    z_1 + z_2 - 2 &= (z_1 - z_2)^2 - (z_1^2-z_1) - (z_2^2-z_2) + 2(z_1z_2-1) \\ 
    &\equiv (z_1 - z_2)^2 \mod I
  \end{align*}
  and the polynomial $z_1 + z_2 - 2$ is said to be $1$-sos modulo~$I$. 
  The certificate consists of the single polynomial $s_1=z_1-z_2$.
\end{example}  
  
It is well-known that not all non-negative polynomials can be
expressed as a sum-of-squares. However, in the particular case when the
ideal is radical and the variety is finite, we can state the following.
\begin{lemma}\label{lem:NonNegVarEquivEllSos} 
  Given a radical ideal $I$ with a finite real variety and a
  polynomial~$f$ with $f(\variety{I}) \subseteq \R$.
  Then $f$ is non-negative on the variety, i.e,
  $\forall z^\ast \in \variety{I} \colon f(z^\ast) \geq 0$, if and only if
  there exists a non-negative integer $\ell$ such that $f$ is $\ell$-sos
  modulo~$I$.
\end{lemma}

\begin{proof}
  Let $f$ be a polynomial that can be expressed as a
  sum-of-squares modulo~$I$, say $f \equiv \sum_{i=1}^{t} s_{i}^{2} \mod I$.
  Since all polynomials in the ideal~$I$ vanish
  on the variety by definition and since the real-valued
  $\sum_{i=1}^{t} s_{i}^{2}$ is
  clearly non-negative, $f$ is non-negative on
  the variety~$\variety{I}$.

  To
  prove the other direction, we recall a well-known argument.
  Suppose we have a polynomial $f$ with $f(z^\ast) \geq 0$
  for all $z^\ast \in \variety{I}$. Suppose further that the
  finite variety~$\variety{I}$ equals $\set{z_1^\ast,\ldots,z_t^\ast}$
  for a suitable~$t$. We now construct $t$ interpolation polynomials $f_i$ for 
  $i \in \set{1, \dots, t}$
  (see~\cite{interp})
  such that
    \begin{align*}
    f_i(z^\ast) &= \begin{cases}
    1 & z^\ast= z_i^\ast,\\
    0 & z^\ast \neq z_i^\ast
    \end{cases}
    \end{align*}
    for all $z^\ast \in \variety{I}$.
    Observe that the square of an interpolating polynomial is again an
    interpolating polynomial. Therefore, the difference polynomial
    \begin{equation}\label{eq:NonNegVarEquivEllSos:diff-poly}
    f(z) - \sum_{i = 1}^t f_i^2(z)f(z_i^\ast)
    \end{equation}
    vanishes on every point~$\set{z_1^\ast,\ldots,z_t^\ast}$ in the variety.
    We now use Hilbert's Nullstellensatz
    (Theorem~\ref{thm:HilbertsNullstellensatz}): Since the ideal~$I$ is
    radical, we apply
    Remark~\ref{remark:HilbertsNullstellensatz} on
    the difference polynomial~\eqref{eq:NonNegVarEquivEllSos:diff-poly}
    and get that it is in~$I$.
    Therefore, if we let
    \begin{equation*}
    s_i =   f_i(z)\sqrt{f(z_i^\ast)},
    \end{equation*}
    we then see that
    \begin{equation*}
    f \equiv \sum_{i=1}^t s_i^2 \mod~I.
    \end{equation*}
\end{proof}
    We observe that the $\ell$ in this case is quite large, since it is
    the degree of the interpolating polynomial $f_i$, which depends on the
    number of points in the variety. However, we will rely on the fact that
    the sum-of-squares representation is not unique, and there may
    exist Positivstellensatz certificates of much lower degree, within
    reach of computation. As we will see in Section~\ref{sec:proof_of_kG=nG_and_kH=nH-1} 
    and~\ref{sec:certificates_kGIsnGMinus1_kHEqnHMinus1}, 
    this does indeed turn out to be the case.

  \section{Vizing's Conjecture as a Sum-of-Squares Problem} \label{sec_poly_I}
  In this section, we describe Vizing's conjecture as a sum-of-squares
  problem. Towards that end, we will first define ideals associated
  with graphs $G$, $H$ and $G \Box H$, and then finally describe an
  ideal/\allowbreak polynomial pair where the polynomial is non-negative on the
  variety of the ideal if and only if Vizing's conjecture is true. We begin by
  creating an ideal where the variety of solutions
  corresponds to graphs with a given number of vertices and size of a
  minimum dominating set.

  The notation underlying all of the definitions in this section---we
  will use it also through the whole article---is as
  follows. 
  Let $n_\cG$ and $k_\cG \leq n_\cG$ be fixed positive integers, and 
  let $\cG$ be the class of graphs on $n_\cG$ vertices with a fixed\footnote{We
    fix the vertices of the dominating set without loss of generality
    as this corresponds to a simple renaming of the vertices.
    Doing this avoids the introduction of additional boolean variables for the
    vertices and reduces the size of the corresponding
    isomorphism group of the variety.
    It is therefore algorithmically favorable.} 
  minimum dominating set $D_\cG$ 
  of size $k_{\cG}$.
  We then turn the various edges ``on'' or ``off''
  (by controlling a boolean variable~$e_{gg'}$ for each possible edge~$\set{g,g'}$)
  such that each point in the
  variety corresponds to a \emph{specific} graph $G \in \cG$.

  \begin{definition}\label{def:IG}
    Set $e_\cG = \set{e_{gg'} \colon \set{g,g'} \subseteq V(\cG)}$.
  The ideal $I_\cG \subseteq P_\cG = \K[e_\cG]$
  is defined by the system of polynomial equations%
  \footnote{Being precise, the ideal~$I_\cG$ is defined by the
    polynomials on the left-hand side of the
    equations~\eqref{eq:eij-fix}, \eqref{eq:domset-fix}
    and~\eqref{eq:kcover-fix}.
    However, we 
    think that the current phrasing provides the better insight and
    is closer to the intended way of thinking for this work.
    
    If one would like to write equations in a formally correct way,
    one first needs to evaluate the polynomial on the left-hand side
    at some suitable point,
    meaning to substitute the variables of the polynomial ring by real values.
    For example, the variable~$e_{gg'} \in P_\cG$ is substituted by
    some (possibly a priori unknown)~$e_{gg'}^\ast \in \R$,
    therefore the polynomial on the left-hand side
    of~\eqref{eq:eij-fix} becomes the equation
    $(e_{gg'}^\ast)^2 - e_{gg'}^\ast = 0$.
    Notation~\ref{notation:poly-vs-evaluated-star-world} is also related
    to this issue and brings the connection to the points in the
    associated variety.

    We will, however, always be precise when the distinction between
    variable and evaluated (starred) form matters.}
  \begin{subequations}
    \label{eq:graphs-G-fix}
    \begin{align}
      e_{gg'}^2 - e_{gg'} &= 0 &
      &\text{for $\set{g,g'}\subseteq V(\cG)$,} \label{eq:eij-fix}\\
      \prod_{g' \in D_\cG} (1 - e_{gg'}) &= 0 &
      &\text{for $g \in V(\cG) \setminus D_\cG$,} \label{eq:domset-fix}\\
      \prod_{g' \in V(\cG) \setminus S} \parentheses[\bigg]{\;\sum_{ g \in S}e_{gg'}} &= 0 &
      &\text{for $S \subseteq V(\cG)$ where $\card{S} = k_\cG-1$.}\label{eq:kcover-fix}
    \end{align}
  \end{subequations}
\end{definition}

\begin{notation}\label{notation:poly-vs-evaluated-star-world}
  Throughout this paper, we will use the following notations: 
  We will use~$z$ for the tuple of variables
  of the polynomial ring~$P$, so $P=\K[z]$.
  When considering the variety $\variety{I}$ associated to an
  ideal~$I \subseteq P$, we use the notation $z^\ast \in \variety{I}$
  for the elements in this variety. 
  
  Note that the polynomial ring variables
  (which are the components of~$z$) correspond bijectively to the
  components of $z^\ast$.
  In particular we will use $e_{gg'}^\ast$ for the component of
  $z^\ast=e_\cG^\ast \in \variety{I_\cG}$ corresponding to the polynomial ring
  variable~$e_{gg'} \in P_\cG$.
\end{notation}

\begin{remark}
  Definition~\ref{def:IG} is meaningful even in the case that
  $n_\cG = 1$. The only vertex must be in the dominating set,
  so~$k_\cG=1$. Pairs~$\set{g,g'}$ cannot be chosen from the
  one-element set~$V(\cG)$, thus the set of variables $e_\cG$ is
  empty. This implies $P_\cG = \K[e_\cG]$ is the polynomial ring
  over~$\K$ in no variables (i.e., isomorphic to~$\K$).

  The polynomials defining the ideal~$I_\cG$ disappear:
  There are no polynomials coming
  from~\eqref{eq:eij-fix} again because of non-existing
  pairs~$\set{g,g'}$. Also, there are no polynomials coming
  from~\eqref{eq:domset-fix} as $V(\cG) \setminus D_\cG$ is empty
  because both $V(\cG)$ and $D_\cG$ consist exactly of the same
  vertex. There is a contribution from~\eqref{eq:kcover-fix} for $S$
  being the empty set, however this is the equation $0=0$, so again no
  true contribution. Thus, the ideal~$I_\cG \subseteq P_\cG$ only
  consists of $0$.

  This, in turn, means that the variety~$\variety{I_\cG}$ is ``full''
  meaning in our particular situation being the set containing the
  empty tuple only.

\end{remark}

\begin{theorem}\label{thm:BijectionVarietyGraphsG}
  The points in the variety $\variety{I_\cG}$ 
  are in bijection to
  the graphs in $\cG$. 
\end{theorem}
\begin{proof}
  For $n_\cG = k_\cG = 1$ this is clearly true, as there is exactly one element in both $\variety{I_\cG}$
  and $\cG$.
    
  For $n_\cG \geqslant 2$ consider any point $z^\ast \in \variety{I_\cG}$.  
  We use Notation~\ref{notation:poly-vs-evaluated-star-world}. 
  Since
  equations~\eqref{eq:eij-fix} turn the edges ``on'' ($e_{gg'}^\ast = 1$) or
  ``off'' ($e_{gg'}^\ast = 0$), the point $z^\ast$ defines a graph~$G$ in
  $n_{\cG}$ vertices. Equations~\eqref{eq:domset-fix} iterate over all the
  vertices inside the set $D_\cG$, and ensure that for each vertex 
  outside the set at least one
  edge from a vertex inside the set to this vertex  
  is ``on''. Therefore, $D_\cG$ is a dominating set. Finally,
  equations~\eqref{eq:kcover-fix} iterate over all sets $S$ of size
  $k_\cG-1$ and ensure that at least one vertex outside~$S$ is
  not incident to any vertex inside~$S$ for any~$S$.
  Therefore, no set $S$ of size $k_\cG - 1$ is a dominating
  set. Thus, every point $z^\ast \in \variety{I_{\cG}}$ corresponds
  to a graph $G$ on $n_\cG$ vertices with a minimum dominating set of
  size $k_\cG$.
  
  With the intuition given above 
  it is straight forward to construct a point in $\variety{I_\cG}$ for 
  a graph on $n_\cG$ vertices with a minimum dominating set of size $k_\cG$.
\end{proof}

Similarly, for fixed positive integers $n_\cH$ and $k_\cH \leq n_\cH$, let
$\cH$ be the class of graphs on $n_\cH$
vertices and a minimum dominating set of size $k_\cH$. 
Again, we fix
the dominating set to some $D_\cH$ to
reduce isomorphisms within the variety.
Furthermore let 
the ideal $I_\cH$ be defined on the polynomial
ring~$P_\cH=\K[e_\cH]$ with $e_\cH = \set{e_{hh'} \colon \set{h,h'} \subseteq V(\cH)}$
such that the
solutions in the variety~$\variety{I_\cH}$ are in bijection to the graphs
in $\cH$.

Next, we define the graph class~$\cG \Box \cH$ and the ideal $I_{\cG \Box \cH}$.
For the above classes~$\cG$ and~$\cH$, the graph class~$\cG \Box \cH$ is
the set of product graphs $G \Box H$ for $G\in\cG$ and $H\in\cH$.
The new variables needed
for the ideal are the variables corresponding to the vertices in the product
graph and indicate if such a vertex is in the dominating set or not. Let
$x_{\cG\Box\cH}=\set[\big]{x_{gh} \colon g \in V(\cG),\, h \in V(\cH)}$ and
set $P_{\cG\Box\cH} = \K[e_\cG \cup e_\cH \cup x_{\cG\Box\cH}]$.

\begin{definition}\label{def:IGH}
  The ideal
  $I_{\cG \Box \cH} \subseteq P_{\cG\Box\cH}$ is defined by the
  system of polynomial equations
  \begin{subequations}
    \label{eq:graphs-GH}
    \begin{align}
      x_{gh}^2 - x_{gh} &= 0, \label{eq:xgh}\\
      \bigl(1-x_{gh}\bigr)
      \parentheses[\bigg]{
        \prod_{\substack{g' \in V(\cG) \\ g' \neq g}}
        \parentheses[\big]{1 - e_{gg'} x_{g'h}}}
      \parentheses[\bigg]{
        \prod_{\substack{h' \in V(\cH) \\ h' \neq h}}
        \parentheses[\big]{1 - e_{hh'} x_{gh'}}} &= 0, \label{eq:xgh_domset}
    \end{align}
  \end{subequations}
 for $g \in V(\cG)$ and $h \in V(\cH)$.
\end{definition}
Observe that we have no restrictions on the edge variables in this
definition. It is only used as a stepping
stone to the final and most important ideal in our polynomial model.

\begin{definition}\label{def:Isos}
  For graph classes~$\cG$ and $\cH$, we set $\Iviz$ to be the ideal
  generated by the elements of $I_\cG$, $I_\cH$ and $I_{\cG\Box\cH}$.
\end{definition}
Note that our definition of~$\Iviz$ depends on the specific parameters~$n_\cG$, 
$n_\cH$, $k_\cG$ and $k_\cH$.
\begin{notation}\label{notation:poly-vs-evaluated-star-world-2}
    Analogously to Notation~\ref{notation:poly-vs-evaluated-star-world} we will 
    write $z^\ast \in \variety{\Iviz}$ for the elements of the variety of 
    $\Iviz$. 
    We will use
    $e_{gg'}^\ast$, $e_{hh'}^\ast$ and $x_{gh}^\ast$ for the component of $z^\ast$
    corresponding to the polynomial ring
    variables~$e_{gg'} \in P_\cG$, $e_{hh'} \in P_\cH$ and 
    $x_{gh} \in P_{\cG\Box\cH}$ respectively.
\end{notation}

\begin{theorem}\label{def:IGHsdp}
  The points in
  the variety $\variety{\Iviz}$ are in bijection to the triples  $(G,H,D)$ where
  $G$ is a graph in $\cG$,
  $H$ is a graph in $\cH$ and
  $D$ is a dominating set of any size of $G\Box H$.
\end{theorem}
\begin{proof}
  We have already demonstrated in Theorem~\ref{thm:BijectionVarietyGraphsG} that
  $\variety{I_\cG}$, $\variety{I_\cH}$ are in bijection to the
  graphs in $n_\cG$, $n_\cH$ vertices with minimum dominating sets of
  size $k_\cG$, $k_\cH$ respectively. It remains to investigate the
  restrictions placed on the $x_{gh}$ variables, which denote whether
  or not the vertex $gh \in V(\cG \Box \cH)$ appears in the dominating set of the product
  graph.
  
  Let $z^\ast \in \variety{\Iviz}$ be a point in the variety.
  We use Notation~\ref{notation:poly-vs-evaluated-star-world}. 
  With the arguments above this point induces a graph $G \in \cG$ and a graph 
  $H \in \cH$.
  Furthermore equations~\eqref{eq:xgh} force the vertex variables
  $x_{gh}$ to evaluate to ``on'' ($x_{gh}^\ast = 1$) or ``off'' 
  ($x_{gh}^\ast = 0$). We define $D$ such that the vertex $gh$ is in $D$ if 
  $x_{gh}^\ast = 1$ and is outside $D$
  otherwise. Equations~\eqref{eq:xgh_domset} ensure that $D$ is a dominating 
  set, because every vertex $gh$ is dominated. 
  Indeed, it is either in the set itself (i.e., $1 - x_{gh}^\ast = 0$),
  or it is adjacent to a vertex in the dominating set $D$ via an edge from
  the underlying graph in $\cG$ or the underlying graph in $\cH$. In
  particular, the edge $\set{g,g'}$ is ``on'' and the vertex $g'h$ is
  in the dominating set ($e_{gg'}^\ast = 1$ and $x_{g'h}^\ast=1$), or the edge 
  $\set{h,h'}$ is ``on'' and the vertex
  $gh'$ is in the dominating set
  ($e_{hh'}^\ast = 1$ and $x_{gh'}^\ast=1$). In either of these cases, the
  vertex $gh$ of the box graph is dominated. Therefore, the points in the variety
  $\variety{\Iviz}$ are in bijection to the graphs in
  $n_\cG$, $n_\cH$ vertices with minimum dominating sets of size
  $k_\cG$, $k_\cH$ respectively, and their corresponding product graph
  with a dominating set $D$ of any size.
  
  With the intuition given above 
  it is straight forward to construct a point in $\variety{I_\cG}$ for graphs 
  $G$, $H$ and a dominating set $D$ in $G\Box H$. 
\end{proof}
Observe that there are no polynomials in $\Iviz$ enforcing minimality
on the dominating set in the product graph. This is essential when we
tie all of the definitions together and model Vizing's
conjecture as a sum-of-squares problem. In particular, we model
Vizing's conjecture as an ideal/polynomial pair, where the polynomial
must be non-negative on the variety associated with the ideal if and
only if Vizing's conjecture is true.

\begin{definition}\label{def:fStar}
  Given the graph classes $\cG$ and $\cH$, define
  \begin{equation*}
    \fviz
    = \biggl(\, \sum_{gh \in V(\cG) \times V(\cH)} x_{gh} \biggr) - k_\cG k_\cH.
  \end{equation*}
\end{definition}
\begin{theorem}\label{proposition:VizSDPConj}
  Vizing's conjecture is true if and only if for all positive integer values of $n_\cG$, $k_\cG$, $n_\cH$ and $k_\cH$ with $k_\cG \le n_\cG$ and $k_\cH \le n_\cH$, 
   $\fviz$ is non-negative on~$\variety{\Iviz}$, i.e.,
  \begin{equation*}
    \forall z^\ast \in \variety{\Iviz} \colon \fviz(z^\ast) \geq 0.
  \end{equation*}
\end{theorem}

\begin{proof}
  Assume that Vizing's conjecture is true, and
  fix the values of $n_\cG$, $k_\cG$, $n_\cH$ and $k_\cH$.
  Therefore, for all graphs
  $G\in\cG$ and $H\in\cH$, we have $\gamma(G \Box H) \geq \gamma(G)\gamma(H)$
  which is equivalent to $\gamma(G \Box H) - k_\cG k_\cH \geq 0$.
  The sum~$\sum_{gh \in V(\cG) \times V(\cH)} x_{gh}^\ast$ coming from $\fviz$
  equals exactly the size of the dominating
  set in the box graph $G\Box H$. Therefore, we have $\fviz(z^\ast) \geq 0$ 
  for all $z^\ast \in \variety{\Iviz}$.

  Similarly,
  if $\fviz(z^\ast) \geq 0$ for all $z^\ast \in \variety{\Iviz}$, every
  dominating set in $G \Box H$ has size at least $k_\cG k_\cH$.
  In particular, the minimum dominating set in $G \Box H$ has
  size at least $k_\cG k_\cH$ and Vizing's conjecture is true.
\end{proof}

\begin{corollary}\label{proposition:VizSDPlsos}
  Vizing's conjecture is true if and only if for all positive integer values of $n_\cG$, $k_\cG$, $n_\cH$ and $k_\cH$ with $k_\cG \le n_\cG$ and $k_\cH \le n_\cH$,
  there exists an integer $\ell$ such that $\fviz$
  is $\ell$-sos modulo $\Iviz$.
\end{corollary}

\begin{proof}
  The ideal $\Iviz$ contains the univariate polynomial $z^2 - z$ for
  each variable. Therefore, by Lemma~\ref{lem_kreuzer_rad}, $\Iviz$ is
  radical. Due to Lemma~\ref{def:IGHsdp}, the variety~$\variety{\Iviz}$ is
  finite and obviously it is real. Therefore, by 
  Lemma~\ref{lem:NonNegVarEquivEllSos}, we know
  that a polynomial is non-negative on $\variety{\Iviz}$, if and only if there
  exists an integer $\ell$ such that the polynomial is $\ell$-sos modulo
  $\Iviz$.
  Hence the result follows from Theorem~\ref{proposition:VizSDPConj}.
\end{proof}

In this section, we have drawn a parallel between Vizing's conjecture
and a sum-of-squares problem. We defined the ideal/polynomial pair
$(\Iviz$, $ \fviz)$ such that
$\fviz(z^\ast) \geq 0$ for all $z^\ast \in \variety{\Iviz}$ if and only if
Vizing's conjecture is true. In the next section, we describe
how to find these Positivstellensatz certificates of non-negativity,
or equivalently, these Positivstellensatz certificates that Vizing's conjecture is
true.

\section{Methodology}
\label{sec:methodology}

\subsection{Overview of the Methodology}
In our approach to Vizing's conjecture we ``partition'' the graphs $G$, $H$ and $G \Box H$ by their sizes (number of vertices) $n_\cG$ and $n_\cH$ and by the sizes of
their dominating sets $k_\cG$ and $k_\cH$. 
Note that we aim
for certificates for all partitions as this would prove the conjecture.
However in the following we present
our method which works for a fixed partition (i.e., for fixed values of
$n_\cG$, $k_\cG$, $n_\cH$ and $k_\cH$), and only later relax this and
generalize to parametrized partitions. 

The outline is as follows:

\begin{itemize}
\item Step 1: Model the graph classes as ideals
\item Step 2: \begin{minipage}[t]{0.8\linewidth}
    Formulate Vizing’s conjecture as 
    sum-of-squares existence question
  \end{minipage}
\item Step 3: Transform to a semidefinite program
\item Step 4: Obtain a numeric certificate by solving the semidefinite program
\item Step 5: Guess an exact certificate
\item Step 6: Computationally verify the certificate
\item Step 7: Generalize the certificate
\item Step 8: Prove correctness of the certificate
\end{itemize}

For fixed values of $n_\cG$, $k_\cG$, $n_\cH$ and $k_\cH$ the first
step is to create the ideal $\Iviz$ as described in
Section~\ref{sec_poly_I}, in particular Definition~\ref{def:Isos}. To
summarize, we create the ideal $\Iviz$ in a suitable polynomial ring
in such a way that the points in
the variety $\variety{\Iviz}$ are in bijection to the triples $(G,H,D)$ 
where $G$ is a graph in $\cG$, $H$ is a graph in $\cH$ 
and $D$ is a dominating set of any size of $G \Box H$; see Theorem~\ref{def:IGHsdp}.
In this polynomial ring there is a variable for each possible edge
of $\cG$ and $\cH$ (indicating whether this edge is
present or not in the particular graphs $G$ and $H$) and for each vertex of 
$\cG \Box \cH$ (indicating whether this vertex is in the
dominating set of $G \Box H$ or not).

The second step is to use the polynomial ring variables mentioned above to 
reformulate Vizing's conjecture: It is true for a fixed partition if the polynomial~$\fviz$ (Definition~\ref{def:fStar}) is
non-negative if evaluated at all points in the variety~$\variety{\Iviz}$ of the
constructed ideal. 
For showing that the polynomial is non-negative, we aim for rewriting it
as a finite sum of squares of polynomials (modulo the ideal~$\Iviz$). If we
find such polynomials, then these polynomials form a certificate for Vizing's
conjecture for the fixed partition.
To be more precise and as already described in
Section~\ref{sec_poly_I}, Vizing's conjecture is true for these fixed values
of $n_\cG$, $k_\cG$, $n_\cH$ and $k_\cH$ if and only if
$\fviz$ is $\ell$-sos modulo $\Iviz$.

In the subsequent Section~\ref{sec:sdp} we describe how to perform step three and do another
reformulation, namely as a semidefinite program. 
Note that in order to do so, we need to have
specified $\ell$, the degree of the certificate.
Note also that in order to prepare the
semidefinite program, we use 
basis polynomials (i.e., special generators)
of the ideals. These
are obtained by computing a Gröbner basis of the ideal;
see~\cite{coxetal} for more information on Gröbner bases.

The fourth step (Section~\ref{sec:step:numerical-certificate}) is now to solve the
semidefinite program.
If the program is infeasible (i.e., there exists no feasible solution),
we increase $\ell$. 
On the other hand, if the program is 
feasible, then we can construct a numeric
sum-of-squares certificate.
As the underlying system of equations---therefore
the future certificate---is quite large, we iterate
the following tasks: Find a numeric solution to the semidefinite
program, find or guess some structure in the solution, use these new
relations to reduce the size of the semidefinite program, and begin
again with solving the new program. This reduces the solution space and therefore
potentially also the size (number~$t$ of summands) of the certificate and
the number of monomials of the $s_i$ from Definition~\ref{definition:ellsos}.
The procedure above goes hand-in-hand with our next step
(Section~\ref{sec:step:exact-certificate}), namely obtaining (one
might call it guessing) an exact certificate out of the numeric
certificate.

Once we have a candidate for an exact certificate, we can check its
validity computationally by summing up the squares and reducing modulo the ideal; see
our step~six described in Section~\ref{sec:step:check-certificate}. 

We want to
point out, that we still consider Vizing's conjecture for a particular partition of graphs.
However, having such certificates for some partitions, one can go for
generalizing them by introducing parametrized partitions of
graphs. Our seventh step in Section~\ref{sec:step:generalize-certificate} provides more
information.

The final step is to 
prove that the newly obtained, generalized certificate candidate is
indeed a certificate; see as well Section~\ref{sec:step:generalize-certificate}.
Further certificates and different generalizations together with their proofs can be found in
Sections~\ref{sec:proof_of_kG=nG_and_kH=nH-1}
and~\ref{sec:certificates_kGIsnGMinus1_kHEqnHMinus1}.

\subsection{Transform to a Semidefinite Program}
\label{sec:sdp}
Semidefinite programming refers to the class of optimization problems 
where a linear function with a symmetric matrix variable is optimized
subject to linear constraints and the constraint that the matrix
variable must be positive semidefinite. A semidefinite program (SDP) can be solved to arbitrary precision in
polynomial time~\cite{VandenbergheBoyd}.
In practice the most prominent methods for solving an SDP
efficiently are interior-point algorithms.
We use the solver MOSEK~\cite{mosek} within MATLAB. For more details on solving 
SDPs and on interior-point algorithms see~\cite{handbook-1}.

It is possible to check whether a polynomial $f$ is 
$\ell$-sos modulo an ideal with semidefinite programming. We refer
to \cite[pg.~298]{sos-instructions} for detailed information and
examples. We will now present how to do so for our setting only. 

Let us first fix (for example, by computing) a reduced Gröbner basis~$B$ of~$\Iviz$ and
fix a non-negative integer~$\ell$. Denote by $v$ the vector of all
monomials in our polynomial ring~$P$ of degree at most $\ell$ which can not be
reduced\footnote{Algorithmically speaking, we say that a polynomial~$f$ is
  reduced modulo the ideal~$I$ if $f$ is the representative of~$f+I$ which is
  invariant under reduction by a reduced Gröbner basis of the ideal~$I$.}
modulo~$\Iviz$ by the Gröbner basis~$B$.
Let $p$ be the length of the vector
$v$. Then $\fviz$ (of Definition~\ref{def:fStar}) is $\ell$-sos modulo $\Iviz$ if and only if there
is a positive semidefinite matrix $X\in \R^{p \times p}$
such that $\fviz$ is equal to
\begin{equation*}
  v^{T}Xv
\end{equation*}
when reduced over~$B$. Hence the SDP we end up with optimizes the matrix
variable
$X\in \R^{p \times p}$ subject to linear constraints that
guarantee $\fviz$ being $v^{T}Xv$ as above. The objective function can be chosen
arbitrarily because any matrix satisfying the constraints is
sufficient for our purpose. More will be said on this later.

If the SDP is feasible, then due to the positive
semidefiniteness we can decompose the solution $X$ into $X = S^{T}S$ for some $S\in \R^{t\times p}$ and $t\le p$.
Subsequently, we
define the polynomial~$s_{i}$ by the $i$th row of $Sv$ and obtain
\begin{equation}\label{eq:sdp-sos-certificate}
 v^{T}Xv = (Sv)^T(Sv) = \sum_{i=1}^t
  s_{i}^2 \equiv \fviz \mod \Iviz.
\end{equation}
Note that the last congruence holds due to the constraints in
the SDP. 
Equation~\eqref{eq:sdp-sos-certificate} then certifies that $\fviz$ can
be written as a sum of squares of
the~$s_i$, and hence, $\fviz$ is $\ell$-sos modulo~$\Iviz$ according to
Definition~\ref{definition:ellsos}. 

If the SDP is infeasible, we have an indication that there is no certificate of
degree~$\ell$. We increase~$\ell$ to $\ell+1$, because $\fviz$ could still be
$(\ell+1)$-sos modulo~$\Iviz$ or posses a certificate of even higher degree.
However, if no new reduced monomials appear in this increment, then
by Lemma~\ref{lem:NonNegVarEquivEllSos} and Theorem~\ref{proposition:VizSDPConj}
Vizing's conjecture does not hold.

\begin{example} \label{example:3232}
    We consider the graph classes $\cG$ and $\cH$ with $n_\cG = 3$,
    $k_\cG = 2$, $n_\cH=3$ and $k_\cH = 2$. Using
    SageMath~\cite{sagemath} we
    construct the ideal $\Iviz$, generated by 32~polynomials in
    15~variables. Again using SageMath, we find a Gröbner basis of
    size~95. 

    First, we check the existence of a 1-sos certificate.
    The vector $v$ for $\ell=1$ has length~$12$, i.e., we set
    up an SDP with a matrix variable $X \in \R^{12 \times 12}$. Imposing
    the necessary constraints to guarantee $\sum_{i=1}^{12} s_{i}^{2} \equiv \fviz
    \mod \Iviz$ leads to $67$~linear equality constraints. Interior-point algorithms detect
    infeasibility of this SDP in less than half a second, this indicates
    that there is no 1-sos certificate. 

    Setting up the SDP for checking the existence of a 2-sos
    certificate results in a problem with a matrix variable of dimension~$67$
    and~$359$ linear constraints. Interior-point algorithms find a
    solution $X$ of this SDP in $0.72$~seconds, this guarantees the
    existence of a numeric $2$-sos certificate for these graph classes. 
\end{example}

\subsection{Obtain a Numeric Certificate}
\label{sec:step:numerical-certificate}


As described in Section~\ref{sec:sdp} above, after solving the SDP we decompose
the solution~$X$. We do so by computing the eigenvalue decomposition
$X=V^TDV$ and then setting $S=D^{1/2}V$, 
where $D$ is the diagonal matrix having the eigenvalues on the main diagonal. 
Since
$X$ is positive semidefinite, all eigenvalues are non-negative and we can
compute $D^{1/2}$ by taking the square root of each of the diagonal entries.
The matrices $X$, $V$ and $D$ are obtained
through numeric computations, hence there might be entries in $D$ that are rather
close to zero but not considered as zero. 
We may try setting these almost-zero eigenvalues to zero, which
reduces the number of polynomials of the sum-of-squares certificate. 

Furthermore, a zero column in $S$ means that the corresponding
monomial is not needed in the certificate. Hence, we may try to
compute a certificate where we remove all monomials corresponding to
almost-zero columns. This can decrease the size of the SDP considerably and a smaller size of the matrix and fewer constraints is favorable for solving the SDP.
Of course, if removing these monomials leads to infeasibility of the SDP, then removing these monomials was not correct.


As already mentioned we can choose the objective function arbitrarily.  
Our
experiments show that different objective functions lead to
(significantly) different solutions. Therefore, we carefully choose a
suitable objective function leading to a ``nice'' solution for each
instance.

\begin{example}
\label{example:3232-2}
    We look again at the case we considered in Example~\ref{example:3232}, that is $\cG$ and $\cH$ with $n_\cG = 3$,
    $k_\cG = 2$, $n_\cH=3$ and $k_\cH = 2$, for which we already obtained an optimal solution $X$ and a numeric $2$-sos certificate.
    
    After computing (numerically) the eigenvalue decomposition $X =
    V^TDV$, we set all almost-zero eigenvalues to zero and compute $S
    = D^{1/2}V$, which results in a $12 \times 67$ matrix, i.e.,
    55~eigenvalues are considered as zero.
    In Figure~\ref{fig:heatmap1} a heat map of matrix 
    $S$ is displayed. It seems unattainable to convert this obtained solution
    to an exact certificate (see Section~\ref{sec:step:exact-certificate}),
    so we take a different path.
    \begin{figure}
    \includegraphics[width=0.9\columnwidth]{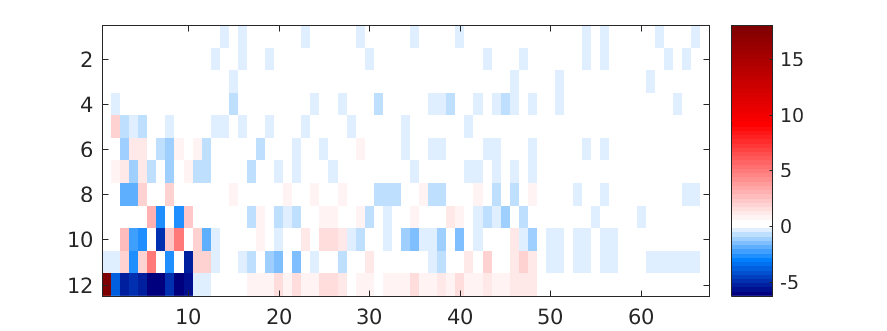}
    \caption{Plotting the entries of the matrix $S$ 
   for $v$ being the vector of all
   monomials in~$P$.
   Every row of $S$ corresponds to one polynomial $s_i$ of the numeric sum-of-squares certificate
   and every column of $S$ corresponds to one monomial in $v$.}
    \label{fig:heatmap1}
    \end{figure}
    
    Using different objective functions and aiming for a certificate
    where only certain monomials appear can lead to results with a
    clearer structure. If the $i$th monomial should not be included,
    we can set the $i$th row and column of $X$ equal to
    zero and obtain another SDP, where we have fewer variables and
    modified constraints. 
    We now try to use only the 19 monomials $1$, $x_{gh}$ and $x_{gh}x_{gh'}$ 
    for all $g \in V(\cG)$ and all $h$, $h' \in V(\cH)$ with $h'\neq h$.
    
    This results in an SDP with a matrix variable of dimension~$19$ and $99$ constraints. The
    SDP can be solved in $0.48$~seconds, and again, we obtain matrix
    $S$ (after setting almost-zero eigenvalues to zero), which now is
    of dimension $4 \times 19$. A heat map is given in Figure~\ref{fig:heatmap2}.
    
    \begin{figure}
    \includegraphics[width=0.9\columnwidth]{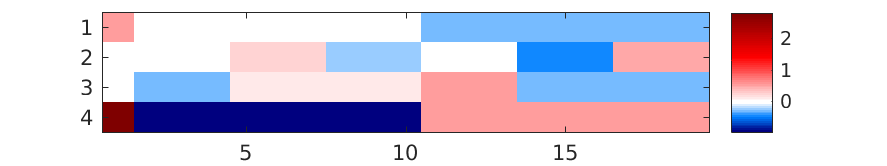}
\caption{Plotting the entries of matrix $S$ as in Figure~\ref{fig:heatmap1}, 
   but now we only allow the coefficients of 19 monomials to be non-zero. 
   The numeric sum-of-squares certificate consists of 4 (number of rows) 
   polynomials in 19 (number of columns) monomials.
   In particular the first three rows correspond to $s_1$, $s_2$ and $s_3$ 
   and the last row corresponds to $s_0$ as given in~\eqref{eq:possible-certificate-3232-difficult}.
}
    \label{fig:heatmap2}
    \end{figure}
    

    As one sees in Figure~\ref{fig:heatmap2}, $S$ has a certain block structure,
    suggesting that
    in each $s_i$ the coefficients of the monomials depend only on the index $g\in V(\cG)$ 
    and there is no dependence on the indices $h \in V(\cH)$.
    Therefore, we aim for a $2$-sos certificate of the form
    $\sum_{i=0}^{n_\cG} s_{i}^2 \equiv \fviz\!\!\mod \Iviz$ with
    \begin{subequations}\label{eq:possible-certificate-3232-difficult}
    \begin{align}
      s_{i} &=  \nu_i
              + \sum_{g \in V(\cG)} \lambda_{g,i} \biggl( \sum_{h \in V(\cH)}  x_{gh}\biggr)
              +\sum_{g \in V(\cG)} \mu_{g,i} \biggl( \sum_{\set{h, h'} \subseteq V(\cH)}  x_{gh}x_{gh'}\biggr)
              \intertext{for $i \in \set{1, \dots, n_\cG}$ and}
              s_{0} &= \alpha
                      + \beta \biggl( \sum_{g \in V(\cG)} \sum_{h \in V(\cH)} x_{gh} \biggr)
                      + \gamma \sum_{g \in V(\cG)} \biggl( \sum_{\set{h, h'} \subseteq V(\cH)} x_{gh}x_{gh'}\biggr),
    \end{align} 
    \end{subequations}
   where the coefficients $\alpha$, $\beta$, $\gamma$, $\nu_{i}$,
   $\lambda_{g,i}$ and $\mu_{g,i}$ are the entries of $S$. However, we only
   have the numeric values
   \begin{equation*}
     S = \begin{small}\left(\begin{array}{rrrrrrr}
0.535&0.011&0.011&0.011&-0.289&-0.289&-0.289\\
0.000&0.000&0.236&-0.236&-0.001&-0.471&0.472\\
-0.000&-0.272&0.136&0.136&0.544&-0.273&-0.272\\
2.778&-0.962&-0.962&-0.962&0.536&0.536&0.536\\
\end{array}\right)
\end{small}

   \end{equation*}
   at hand and it is not obvious how to guess suitable
   exact numbers from it. 
   In contrast, looking at the values
   \begin{equation*}
     X = \begin{small}\left(\begin{array}{rrrrrrr}
8.000&-2.667 &-2.667&-2.667& 1.333& 1.333&1.333\\
-2.667&1.000 &0.889 & 0.889&-0.667&-0.444&-0.444\\
-2.667&0.889 &1.000 & 0.889&-0.444&-0.667&-0.444\\
-2.667&0.889 &0.889 & 1.000&-0.444&-0.444&-0.667\\
 1.333&-0.667&-0.444&-0.445& 0.667& 0.222&0.222\\
 1.333&-0.444&-0.667&-0.445& 0.222& 0.667&0.222\\
 1.333&-0.444&-0.444&-0.667& 0.222& 0.222&0.667\\
\end{array}\right)
\end{small}
,
   \end{equation*}
   it seems almost
   obvious which simple algebraic numbers the entries of $X$ could be,
   for example $0.667$ could be $2/3$.
    We will use that in the following section.
\end{example}

\subsection{Guess an Exact Certificate}
\label{sec:step:exact-certificate}

We now have a guess for the structure of the certificate, but coefficients that are simple algebraic numbers are hard to
determine from the numbers in $S$. On the other hand, the exact numbers in $X$
seem to be rather obvious so we go back to the relation
$X=S^{T}S$. It implies that if we fix two monomials then the inner product of the vectors of the coefficients of these monomials in all the $s_i$ has to be equal to the corresponding number in $X$.

Setting up a system of equations using all possible inner products, we may
obtain a solution to this system. This solution determines the coefficients in
the certificate (and the certificate might be simplified even further).

\begin{example}
\label{example:3232-3}

    We continue Example~\ref{example:3232}, that is we consider the graph classes $\cG$ and $\cH$ with $n_\cG = 3$,
    $k_\cG = 2$, $n_\cH=3$ and $k_\cH = 2$. 
  
    The exact numbers in $X$ given in Example~\ref{example:3232-2} can be guessed easily.
    In fact, if this guess for $X$ is correct, every choice of $S$ such that
    $S^TS = X$ gives a certificate. Using the relation $S^TS = X$ we set up a
    system of equations on the parameters of~\eqref{eq:possible-certificate-3232-difficult}. 
    To be more precise, let $\lambda_g = (\lambda_{g,i})_{i=1,\dots,n_\cG}$,
    $\mu_g = (\mu_{g,i})_{i=1,\dots,n_\cG}$ and $\nu =
    (\nu_{i})_{i=1,\dots,n_\cG}$. Then we can define the vectors $a =
    \binom{\nu}{\alpha}$, $b_g = \binom{\lambda_g}{\beta}$ and $c_g =
    \binom{\mu_g}{\gamma}$, and $S^TS = X$ (together with the guessed values for $X$) implies that
    \begin{align*}
    \langle a,a \rangle     &= 2(n_\cG - 1)^2, & & \\
    \langle a,b_g \rangle   &= -\tfrac{4}{3}(n_\cG - 1),&    \langle a,c_g \rangle &= \tfrac{2}{3}(n_\cG - 1),\\
    \langle b_g,b_g \rangle &= 1 ,          &    \langle b_g,b_{g'} \rangle &= \tfrac{8}{3},\\
    \langle c_g,c_g \rangle &= \tfrac{6}{9}, &    \langle c_g,c_{g'} \rangle &= \tfrac{2}{9},\\
    \langle b_g,c_g \rangle &= -\tfrac{6}{9},&    \langle b_g,c_{g'} \rangle &= -\tfrac{4}{9}
    \end{align*}
    has to hold for each $g\in V(\cG)$, where $\langle\cdot,\cdot\rangle$
    denotes the standard inner product.
    Under the assumption that our guess for $X$ was correct,
    each solution to this system of equations leads to a valid sum-of-squares
    certificate~\eqref{eq:possible-certificate-3232-difficult}. 

    We want a sparse certificate and the numeric solution suggests that $\nu_2 = \nu_3 = 0$ holds, so we try to obtain a solution with also $\nu_1 = 0$ (even though the numeric solution does not fit into that setting).
    Using these values, the equations involving the vector $a$ determine the exact values for $\alpha$, $\beta$ and $\gamma$ as $\alpha = \sqrt{2}(n_\cG - 1)$, $\beta = -\frac{2}{3}\sqrt{2}$ and $\gamma = \frac{1}{3}\sqrt{2}$. 
    With that, the system of equations simplifies to
    \begin{align*}
    \langle \lambda_g,\lambda_g \rangle &= \tfrac{1}{9},  & \langle \lambda_g,\lambda_{g'} \rangle &= 0,\\
    \langle \mu_g,\mu_g \rangle         &= \tfrac{4}{9},  & \langle \mu_g,\mu_{g'} \rangle &= 0,\\
    \langle \lambda_g,\mu_g \rangle     &= -\tfrac{2}{9}, & \langle \lambda_g,\mu_{g'} \rangle &= 0.
    \end{align*}
    Calculating $\sum_{i=1}^{n_\cG} s_i^2$ we find out that, due to the system of equations, the sum-of-squares simplifies to
    \begin{align*}
    \sum_{i=1}^{n_\cG} s_i^2 = \frac{1}{9}\sum_{g \in V(\cG)}\biggl( 
    \sum_{h \in V(\cH)}  x_{gh} - 2 \sum_{\set{h, h'} \subseteq V(\cH)}  x_{gh}x_{gh'}
    \biggr)^2.
    \end{align*}
    Hence, if~\eqref{eq:possible-certificate-3232-difficult} is a sum-of-squares certificate then also
    \begin{subequations}\label{eq:possible-certificate-3232-easy}
    \begin{align}
      s_{0} &=  \alpha
              + \beta \biggl( \sum_{g \in V(\cG)} \sum_{h \in V(\cH)} x_{gh} \biggr)
              + \gamma \biggl(\sum_{g \in V(\cG)}  \sum_{\set{h, h'} \subseteq V(\cH)} x_{gh}x_{gh'}\biggr),\\    
    s_{g} &= \frac{1}{3}\biggl(  \sum_{h \in V(\cH)}  x_{gh} -  2 \sum_{\set{h, h'} \subseteq V(\cH)}  x_{gh}x_{gh'} \biggr)
    \quad\text{for $g \in V(\cG)$},
    \end{align}
  \end{subequations}
    where $\alpha = \sqrt{2}(n_\cG - 1)$, $\beta = -\frac{2}{3}\sqrt{2}$ and $\gamma = \frac{1}{3}\sqrt{2}$ is a sum-of-squares certificate.
\end{example}

To close this section let us highlight once more that we use the SDP and its 
solution 
to find out which monomials are used in the 
certificate and to obtain a structure of their coefficients. 
In particular we do not need a solution of the SDP of high precision, 
so 
solving the SDP is not a bottleneck in this example.
It will turn out that this is also true for all other examples we consider.

\subsection{Computationally Verify the Certificate}
\label{sec:step:check-certificate}

When a certificate is conjectured, it is straightforward to verify
it computationally via SageMath~\cite{sagemath}. To do so, it is necessary to compute the
Gröbner basis of $\Iviz$. Observe that at this point, semidefinite programming is no longer needed.

\begin{example}
 We computationally verified (using SageMath) the certificate derived in
 Example~\ref{example:3232-3} for the graph classes $\cG$ and $\cH$ with $n_\cG = 3$,
    $k_\cG = 2$, $n_\cH=3$ and $k_\cH = 2$. 
\end{example}

\subsection{Generalize the Certificate and Prove Correctness}
\label{sec:step:generalize-certificate}

In Sections \ref{sec:sdp} to \ref{sec:step:check-certificate}, we presented a methodology for obtaining a sum-of-squares certificate for graph classes $\cG$ and $\cH$ with fixed parameters $n_\cG$,
$k_\cG$, $n_\cH$ and $k_\cH$. Assuming that the previously found pattern generalizes, one can iterate the steps outlined above to obtain certificates for larger classes of graphs.

\begin{example}
    We want to generalize the certificate for the graph classes $\cG$ and $\cH$ with $n_\cG = 3$,
    $k_\cG = 2$, $n_\cH=3$ and $k_\cH = 2$ to the case $k_\cG = n_\cG - 1\ge 
    1$, $n_\cH = 3$ and $k_\cH = 2$ for any $n_\cG \ge 2$.

    Solving the SDP for the cases $n_\cG=4$ and $n_\cG=5$ again yields nicely structured matrices and in fact, all the calculations done for the case $n_\cG=3$ (which we already wrote down parametrized by $n_\cG$ above) go through.
   
    Hence, we are able to generalize the sum-of-squares certificate~\eqref{eq:possible-certificate-3232-easy} in the following way.
\end{example}
   
    \begin{theorem} \label{thm:sosCertificate_kGEqnGMinus1_kH2nH3}
      For $k_\cG = n_\cG - 1 \ge 1$, $n_\cH = 3$ and $k_\cH = 2$
      Vizing's conjecture is true as the polynomials
        \begin{align*}
          s_{0} &=  \alpha
                  + \beta \biggl( \sum_{g \in V(\cG)} \sum_{h \in V(\cH)} x_{gh} \biggr)
                  + \gamma \biggl(\sum_{g \in V(\cG)}  \sum_{\set{h, h'} \subseteq V(\cH)} x_{gh}x_{gh'}\biggr)   
            \intertext{and}
            s_{g} &= \frac{1}{3}\biggl(  \sum_{h \in V(\cH)}  x_{gh} -  2 \sum_{\set{h, h'} \subseteq V(\cH)}  x_{gh}x_{gh'} \biggr)
                    \quad\text{for $g \in V(\cG)$},
        \end{align*}
        where $\alpha = \sqrt{2}(n_\cG-1)$, $\beta = -\frac{2}{3}\sqrt{2}$
        and $\gamma = \frac{1}{3}\sqrt{2}$,
        are a sum-of-squares certificate with degree~$2$ of $\fviz$.
    \end{theorem}
    The proof will be given later on after introducing some more auxiliary results; see Section~\ref{sec:certificates_kGIsnGMinus1_kHEqnHMinus1}.
    Of course, once having the theorem above, it
    can be verified computationally for particular parameter values, for example
    for $k_\cG=4$ and $n_\cG=5$, where the computation of a Gröbner basis
    is feasible.

\subsection{Summary}
In this section we saw by an example how to use our
machinery combined with clever guessing in order to obtain sum-of-squares
certificates for proving that Vizing's conjecture holds for fixed
values of $n_\cG$, $k_\cG$, $n_\cH$ and $k_\cH$, and how this can be used to
obtain certificates for a less restricted set of parameters. We will use the next
sections in order to present further certificates and generalizations that we found using
our method and for which we were able to prove
correctness.

\section{Exact Certificates for \texorpdfstring{$k_\cG = n_\cG$}{kG=nG} and \texorpdfstring{$k_\cH = n_\cH-d$}{kH=nH-d}}
\label{sec:proof_of_kG=nG_and_kH=nH-1}

In this section we deal with certificates for the case $k_\cG = n_\cG$ and 
$k_\cH = n_\cH-d$. Towards this end we will first prove several auxiliary
results in Section~\ref{sec:sigma-calculus}.
Next we present and prove certificates for $d = 0$, $d=1$ and
$d=2$ in 
Sections~\ref{sec:sosCertificaten_kGeqnG_kHeqnH_easy},~\ref{sec:sosCertificaten_kGeqnG_kHeqnHminus1_easy}
 and~\ref{sec:sosCertificaten_kGeqnG_kHeqnHminus2_easy}.
Then in Section~\ref{sec:sosCertificaten_kGeqnG_kHeqnHMinusd} we will see how 
this brings insight on the structure of
the certificates. We are therefore able to formulate
a conjecture on the structure of the certificate for general $d$ 
and also present a strategy for proving it. This will be complemented
by a more computational approach for checking the conjecture for a
given value~$d$; in particular we will prove the conjecture for $d=3$
and $d=4$ with the help of SageMath~\cite{sagemath}.

\subsection{Auxiliary Results: Sigma Calculus}
\label{sec:sigma-calculus}

In this section we will develop the machinery needed to prove the
correctness of our (exact) certificates. It will turn out that the key is to be
able to do operations with certain symmetric polynomials, 
which will be introduced in 
Definition~\ref{def:SigSubset_g_i}.
Another important tool will be again
Theorem~\ref{thm:HilbertsNullstellensatz}, 
Hilbert's Nullstellensatz. It's implications formulated
as~Remark~\ref{remark:HilbertsNullstellensatz}
will be used repeatedly, for example in the proof of the following first lemma.

\begin{lemma}\label{lem:eggPrimeZero_kGeqnG}
  \label{eq:1MinuseG}
  Let $k_\cG = n_\cG \geq 1$. Then $e_{gg'} \in I_\cG \subseteq \Iviz$
  holds for all $\set{g,g'} \subseteq V(\cG)$.
\end{lemma}

Translating this lemma in terms of congruence relations, we have
$e_{gg'} \equiv 0 \mod I_\cG$ and $e_{gg'} \equiv 0 \mod \Iviz$ for
all $\set{g,g'} \subseteq V(\cG)$. 

Let us briefly consider Lemma~\ref{lem:eggPrimeZero_kGeqnG} from a graph theoretic 
point of view. 
Due to Theorem~\ref{thm:BijectionVarietyGraphsG} the points in the variety of 
$I_\cG$ are in bijection to the graphs in $\cG$, which are the graphs on 
$n_\cG$ 
vertices with domination number $k_\cG=n_\cG$. 
It is easy to see that such graphs can not have any edges, because otherwise the 
domination number would be strictly less than $n_\cG$.
Hence $e_{gg'}^\ast = 0$ holds for all points $z^\ast$ in the variety of 
$\Iviz$, 
when we use Notation~\ref{notation:poly-vs-evaluated-star-world}. 
This intuitively justifies Lemma~\ref{lem:eggPrimeZero_kGeqnG} by 
graph theoretical considerations.

\begin{proof}[Proof of Lemma~\ref{lem:eggPrimeZero_kGeqnG}]
  For $k_\cG = n_\cG = 1$ there is no $\set{g,g'} \subseteq V(\cG)$,
  so there is nothing to prove.

  For each $\set{g,g'} \subseteq V(\cG)$, we apply
  Hilbert's Nullstellensatz on the polynomial $f = e_{gg'}$.

  We use Notation~\ref{notation:poly-vs-evaluated-star-world}, and
  let $z^\ast \in \variety{I_\cG}$, i.e., $z^\ast$ is a
  common zero of \eqref{eq:eij-fix}, \eqref{eq:domset-fix} and
  \eqref{eq:kcover-fix}. Then clearly $e^\ast_{gg'} \in \set{0,1}$ due to
  \eqref{eq:eij-fix}.
  Furthermore $k_\cG = n_\cG$ implies that
  \eqref{eq:kcover-fix} simplifies to the equations
  \begin{equation*}
    \sum_{\substack{g \in V(\cG) \\ g \neq g'}} e^\ast_{gg'} = 0
    \quad\text{for $g' \in V(\cG)$}.
  \end{equation*} 
  Therefore $e^\ast_{gg'} = 0$ for all
  $\set{g,g'} \subseteq V(\cG)$. Hence $z^\ast$ is also a zero
  of $f=e_{gg'}$ and Hilbert's Nullstellensatz
  (Theorem~\ref{thm:HilbertsNullstellensatz} and
  Remark~\ref{remark:HilbertsNullstellensatz}) implies $f = e_{gg'} \in I_\cG$.
\end{proof}

\begin{lemma}\label{lem:prodxghj_kGeqnG_kHeqnHMinusi}
  Let $k_\cG = n_\cG \geq 1$ 
    and $k_\cH = n_\cH - d \geq 1$ 
    for some $d \geq 0$.
  Moreover, let $g \in V(\cG)$ and $T \subseteq V(\cH)$ be a subset
  of size~$\card{T}=d+1$.
  Then
  \begin{equation}\label{eq:prod1MinusxghInIviz}
    \prod_{h \in T} (1-x_{gh}) \in \Iviz.
  \end{equation}
  Moreover, we have
  \begin{equation}\label{eq:prodxghj_kGeqnG_kHeqnHMinusi}
    \prod_{h \in T} x_{gh} \equiv 
    \sum_{r=0}^{d} (-1)^{d+r}\sum_{\substack{U \subseteq T\\\card{U}=r}}\;\prod_{h \in U} x_{gh}
    \mod \Iviz.
  \end{equation}
\end{lemma}

Note that also Lemma~\ref{lem:prodxghj_kGeqnG_kHeqnHMinusi} can be 
justified intuitively from the graph theoretic perspective. 
According to Theorem~\ref{def:IGHsdp}, a point in the variety of $\Iviz$ 
corresponds to two graphs $G$ and $H$ with $n_\cG$ and $n_\cH$ vertices and 
domination numbers~$k_\cG=n_\cG$ and~$k_\cH=n_\cH - d$ respectively, and a 
dominating set $D$ in $G\Box H$.
Due to Lemma~\ref{lem:eggPrimeZero_kGeqnG} there is no edge in $G$. 
Hence each vertex $gh$ in $G\Box H$ either must be in the 
dominating set $D$ itself, 
or there must be a vertex $h'\in V(H)$ 
such that $gh'$ is in $D$  and the edge $\set{h,h'}$ is in $E(H)$.
In other words, for fixed $g \in V(G)$, the vertices $h \in V(H)$ with $x^\ast_{gh} 
= 1$ have 
to form a dominating set in $H$. Since every dominating set in $H$ has at least 
$k_\cH = n_\cH - d$ vertices, at most $d$ vertices are not in a dominating 
set. Therefore, whenever we choose $d+1$ vertices from $V(H)$, 
at least one vertex has to be in $D$.
Equivalently, in every set $T$ of $d+1$ vertices of $V(H)$ there is at least one 
vertex $h$ with $x^\ast_{gh} = 1$, 
which is stated in~\eqref{eq:prod1MinusxghInIviz}.

\begin{proof}[Proof of Lemma~\ref{lem:prodxghj_kGeqnG_kHeqnHMinusi}]
  We use Hilbert's Nullstellensatz for $f = \prod_{h \in T} (1-x_{gh})$.

  Let $z^\ast \in \variety{\Iviz}$, i.e., $z^\ast$ is a common zero of
  \eqref{eq:eij-fix}, \eqref{eq:domset-fix} and \eqref{eq:kcover-fix}
  for both $\cG$ and $\cH$, and of \eqref{eq:xgh} and
  \eqref{eq:xgh_domset}. Note that we use 
  Notation~\ref{notation:poly-vs-evaluated-star-world-2}.

  Let us consider the second factor of~\eqref{eq:xgh_domset}.
  If $n_\cG = 1$, then there is no $g' \neq g \in V(\cG)$,
  so this product is empty and equals~$1$. If $n_\cG \ge 2$, then
  $e_{gg'}^\ast = 0$ (the component of $z^\ast$ corresponding
  to $e_{gg'}$) for all $g' \in V(\cG)$ because of
  Lemma~\ref{lem:eggPrimeZero_kGeqnG}, and the product equals~$1$ again.
  Hence \eqref{eq:xgh_domset}
  implies
  \begin{equation}
    \label{eq:IdealEquationSimplifiedProofProd1MinusxghIsZero}
    \bigl(1-x^\ast_{gh}\bigr)
    \parentheses[\bigg]{
    \prod_{\substack{h' \in V(\cH) \\ h' \neq h}}
    \parentheses[\big]{1 - e^\ast_{hh'} x^\ast_{gh'}}} = 0
    \quad\text{for $h \in V(\cH)$}. 
  \end{equation}
  Furthermore $e_{hh'}^\ast \in \set{0,1}$ for all
  $\set{h,h'} \subseteq V(\cH)$ because of \eqref{eq:eij-fix}, and
  $x_{gh}^\ast \in \set{0,1}$ for all $h \in V(\cH)$ due to
  \eqref{eq:xgh}.
    
  Assume that $z^\ast$ is not a zero of $f$. Then clearly
  $x^\ast_{gh}=0$ for all $h \in T$. In particular,
  \eqref{eq:IdealEquationSimplifiedProofProd1MinusxghIsZero}
  simplifies to
  \begin{equation*}
    \prod_{\substack{h' \in V(\cH) \\ h' \neq h}}
    \parentheses[\big]{1 - e^\ast_{hh'} x^\ast_{gh'}} =0
    \quad\text{for $h \in T$}.
  \end{equation*} 
  Therefore, for each $h \in T$, there is a $h' \in V(\cH)$ such that
  $e_{hh'}^\ast=1$ and $x_{gh'}^\ast = 1$. As $x_{gh}^\ast=0$ for all
  $h \in T$, we have $h' \not\in T$.
  
  If $n_\cH = 1$, then $d = 0$ and $\card{T} = 1$. 
  But then $V(\cH) \setminus T$ is empty, so no choice for $h'$ is left,
  a contradiction.
  If $n_\cH \ge 2$, then with $S = V(\cH) \setminus T$ the equation \eqref{eq:kcover-fix} for $\cH$
  simplifies to
  \begin{equation*}
    \prod_{h \in T} \parentheses[\bigg]{\;\sum_{h'' \in S} e^\ast_{h''h}} = 0.
  \end{equation*}
  For each $h \in T$, the $h'$ (defined above) is in~$S$, so the
  summand $e^\ast_{h''h} = e^\ast_{hh'}=1$ for $h''=h'$. All other
  summands are either $0$ or $1$,
  hence each sum
  $\sum_{h'' \in S} e^\ast_{h''h}$ is at least one, so in particular
  non-zero. This is again a contradiction. 
  
  Hence for all $n_\cG \ge 1$ 
  our assumption was wrong, so $z^\ast$ is a zero of $f$.
  Now, Hilbert's Nullstellensatz
  (Theorem~\ref{thm:HilbertsNullstellensatz} and
  Remark~\ref{remark:HilbertsNullstellensatz}) implies $f \in \Iviz$, 
  so~\eqref{eq:prod1MinusxghInIviz} is satisfied.
  
  Furthermore, \eqref{eq:prod1MinusxghInIviz} above can be rewritten as
  \begin{equation*}
    \prod_{h \in T} (1-x_{gh}) \equiv 0 \mod \Iviz.
  \end{equation*}
  Therefore, the congruence~\eqref{eq:prodxghj_kGeqnG_kHeqnHMinusi}
  follows from the fact that
  \begin{align*}
    \prod_{h \in T} (1-x_{gh}) 
    =
    \sum_{r=0}^{d+1} (-1)^{r}
    \sum_{\substack{U \subseteq T\\\card{U}=r}}\;\prod_{h \in U} x_{gh}
    \end{align*}
    holds.
\end{proof}

\begin{remark}
   \label{rem:Prod_of_xgh}
  In particular 
  for $k_\cG = n_\cG \geq 1$ and $k_\cH = n_\cH \ge 1$,
  Lemma~\ref{lem:prodxghj_kGeqnG_kHeqnHMinusi} implies
  \begin{equation*}   
  x_{gh} \equiv 1  \mod  \Iviz
  \end{equation*}
  for all  $g \in V(\cG)$ and all $h \in V(\cH)$.
  For $k_\cH = n_\cH - 1$,
  Lemma~\ref{lem:prodxghj_kGeqnG_kHeqnHMinusi} implies
    \begin{equation*}   
    x_{gh}x_{gh'} \equiv x_{gh} + x_{gh'} - 1  \mod  \Iviz
    \end{equation*}
    for all  $g \in V(\cG)$ and all $\set{h, h'} \subseteq V(\cH)$.
    For $k_\cH = n_\cH - 2$,
    Lemma~\ref{lem:prodxghj_kGeqnG_kHeqnHMinusi} implies
    \begin{equation*}
    x_{gh}x_{gh'}x_{gh''} \equiv x_{gh}x_{gh'} + x_{gh'}x_{gh''} + x_{gh}x_{gh''} - x_{gh} - x_{gh'} - x_{gh''} + 1  \mod  \Iviz
    \end{equation*}
    for all  $g \in V(\cG)$ and all $\set{h, h',h''} \subseteq V(\cH)$.
\end{remark}

Note that from a high-level point of view, if $k_\cH = n_\cH - d$,
then Lemma~\ref{lem:prodxghj_kGeqnG_kHeqnHMinusi} allows us to rewrite
particular products of $d+1$ terms as a sum of products of at most $d$
terms and therefore to reduce the degree of polynomials.

To continue and in order to apply the above findings, we introduce the
following polynomials.

\begin{definition}\label{def:SigSubset_g_i}
  Let $g \in V(\cG)$ and $i$ be a non-negative integer. We define
    \begin{align*}
    \SigSubset{i} = \sum_{\substack{S \subseteq V(\cH)\\\card{S}=i}}\;\prod_{h \in S} x_{gh}.
    \end{align*}
\end{definition}

In a classical setting the polynomial $\SigSubset{i}$ is the
elementary symmetric polynomial of degree $i$ in $n_\cH$ variables.
In the following we will investigate the arithmetic of the $\SigSubset{i}$ over the ideal $\Iviz$ and 
aim for
getting nice expressions for products of $\SigSubset{i}$.

\begin{lemma}\label{lem:ProdOfSumsAsSumOfSums}
    Let $k_\cG$, $n_\cG$, $k_\cH$, $n_\cH \geq 1$ and let $i \geq j$. Then
    \begin{align*}
    \SigSubset{i}\,\SigSubset{j} \equiv \sum_{r=0}^{\min\set{j,n_\cH-i}} \binom{j}{r}\binom{i+r}{j} \SigSubset{i+r} \mod \Iviz
    \end{align*}
    holds.
\end{lemma}

Note that we can extend the summation range to $0 \le r \le j$ as
$\SigSubset{i}=0$ for all $i>n_\cH$. This makes the formulation of the lemma
completely independent of the parameters $k_\cG$, $n_\cG$, $k_\cH$ and $n_\cH$.
Moreover,
we will see in the proof that we actually only need generators~$x^2-x$
in the ideal, making the lemma valid in a more general setting.

\begin{remark}\label{remark:ProdOfSumsAsSumOfSums:small-values}
  As needed later, we state Lemma~\ref{lem:ProdOfSumsAsSumOfSums} for
  some particular values of $i$ and $j$. We have
  \begin{alignat*}{2}
    \SigSubset{1}^2
    &\equiv \SigSubset{1} + 2\SigSubset{2} &&\mod \Iviz, \\
    \SigSubset{2}\SigSubset{1}
    &\equiv 2\SigSubset{2} + 3\SigSubset{3} &&\mod \Iviz, \\
    \SigSubset{2}^2
    &\equiv \SigSubset{2} + 6\SigSubset{3} + 6\SigSubset{4} &&\mod \Iviz.
  \end{alignat*}
\end{remark}

Now we come back to the proof of Lemma~\ref{lem:ProdOfSumsAsSumOfSums}.
In the following, we use the phrase \emph{power products} to
refer to products of powers of variables
with non-negative exponent, or in other words, to the summands of a
polynomial without their coefficient. 

\begin{proof}[Proof of Lemma~\ref{lem:ProdOfSumsAsSumOfSums}]
  We start with a couple of remarks. All summands of $\SigSubset{i}$
  and $\SigSubset{j}$ have degree $i$ and $j$ respectively. Hence all summands in the
  product $\SigSubset{i}\,\SigSubset{j}$ are summands of degree
  $i+j$.  Furthermore, whenever two summands in $\SigSubset{i}$ and
  $\SigSubset{j}$ contain the same factor $x_{gh}$, a resulting
  factor~$x_{gh}^2$ can be reduced to $x_{gh}$ over $\Iviz$ because of
  \eqref{eq:xgh}. Therefore, all summands in
  $\SigSubset{i}\,\SigSubset{j}$ are square-free and will have degree
  at least $i$ and at most $i+j$. Clearly the degree is also bounded
  by $n_\cH$. Moreover $\SigSubset{i}$ and $\SigSubset{j}$ are
  symmetric in $h \in V(\cH)$. By all these considerations, it is
  therefore possible to write
  \begin{equation}\label{eq:ProdOfSumsAsSumOfSums:shape}
    \SigSubset{i}\,\SigSubset{j} \equiv
    \sum_{r=0}^{\min\set{j,n_\cH-i}} \delta_r \SigSubset{i+r} \mod \Iviz
  \end{equation}
  for some coefficients~$\delta_r \in \Z$. In fact, these coefficients
  are non-negative.

  For the following considerations, we always reduce modulo~$\Iviz$,
  therefore reducing exponents of monomials larger than one to
  exponents exactly one. So let us fix a power product $x_i$ of
  $\SigSubset{i}$ of degree~$i$ 
  (i.e., $x_i = \prod_{h \in S} x_{gh}$ for some $S$ with $\card{S} = i$)
  and count power products $x_j$ of
  $\SigSubset{j}$ so that the product $x_ix_j$ is of degree $i+r$ (as
  said, after reducing the power product over $\Iviz$). Apparently,
  there have to be $r$ factors in $x_j$ which are not factors of
  $x_i$; there are $\binom{n_\cH-i}{r}$ possible such choices. The
  remaining $j-r$ factors of $x_j$ have to be among the factors of
  $x_i$, hence there are $\binom{i}{j-r}$ possible choices. Finally,
  we note that there are $\binom{n_\cH}{i}$ choices for the fixed
  power product~$x_i$ above.

  In total, expanding the product $\SigSubset{i}\,\SigSubset{j}$
  results in a sum of $\binom{n_\cH-i}{r}\binom{i}{j-r}\binom{n_\cH}{i}$
  power-products of degree $i+r$ for each~$r$.
  We now collect these power products
  to determine the coefficients~$\delta_r$ of the corresponding summand. Each
  sum~$\SigSubset{i+r}$ consists of $\binom{n_\cH}{i+r}$
  power products of degree~$i+r$. Hence and due to the
  representation~\eqref{eq:ProdOfSumsAsSumOfSums:shape}, we have
  \begin{equation*}
    \delta_r
    = \binom{n_\cH-i}{r}\binom{i}{j-r}\binom{n_\cH}{i} \bigg/ \binom{n_\cH}{i+r}
    = \frac{(i+r)!}{r!\,(j-r)!\,(i-j+r)!} = \binom{j}{r}\binom{i+r}{j},
  \end{equation*}
  which completes the proof.
\end{proof}

Lemma~\ref{lem:ProdOfSumsAsSumOfSums}
allows us to replace products of our symmetric
polynomials~$\SigSubset{i}$ by sums.
This will become very handy in
proving certificates.

We now go back to our particular set-up with $k_\cH = n_\cH - d$.  The
next important ingredient is the following lemma, which allows us to
reduce some~$\SigSubset{d+j+1}$ of ``high'' degree.

\begin{lemma}\label{lem:reduceSigTooLarge}
  Let $k_\cG = n_\cG \geq 1$ and $k_\cH = n_\cH - d \geq 1$ for some
  $d \geq 1$. Let $j$ be a non-negative integer. Then
  \begin{equation*}
    \SigSubset{d+j+1}
    \equiv \binom{n_\cH}{d+j+1} \sum_{r=0}^{d} 
    \frac{\binom{d+1}{r}}{\binom{n_\cH}{j+r}} (-1)^{d+r}
     \SigSubset{j+r}
    \mod \Iviz
  \end{equation*}
  holds.
\end{lemma}

\begin{proof}
  For each $\prod_{h \in S} x_{gh}$ in the definition of
  $\SigSubset{d+j+1}$, we fix an arbitrary partition $S = T \cup W$ with $T$ and $W$
  disjoint in a way that $\card{T}=d+1$ and $\card{W}=j$. Therefore, we
  obtain
  \begin{equation*}
    \SigSubset{d+j+1} =
    \sum_{\substack{S \subseteq V(\cH)\\\card{S}=d+j+1}}\;\prod_{h \in S} x_{gh}
    = \sum_{\substack{T \cup W \subseteq V(\cH)\\ \text{$T$, $W$ disjoint} \\\card{T}=d+1\\\card{W}=j}}
    \biggl(\,\prod_{h \in T} x_{gh} \biggr)
    \biggl(\,\prod_{h' \in W} x_{gh'} \biggr).
  \end{equation*}
  With
  Lemma~\ref{lem:prodxghj_kGeqnG_kHeqnHMinusi}, we can reformulate
  this to
  \begin{equation}
    \label{eq:sumsumsumOfSigmaionej}
    \SigSubset{d+j+1} \equiv
    \sum_{\substack{T \cup W \subseteq V(\cH)\\ \text{$T$, $W$ disjoint} \\\card{T}=d+1\\\card{W}=j}}\;
    \sum_{r=0}^{d} (-1)^{d+r}
    \sum_{\substack{U \subseteq T\\\card{U}=r}}
    \biggl(\,\prod_{h \in T} x_{gh} \biggr)
    \biggl(\,\prod_{h' \in W} x_{gh'} \biggr)
    \mod \Iviz.
  \end{equation}
  
  Due to symmetry in $h\in V(\cH)$, minimum degree $j$ and maximum
  degree $d+j$ of the right-hand side, we can rewrite \eqref{eq:sumsumsumOfSigmaionej} to a
  representation
  \begin{equation*}
    \SigSubset{d+j+1} \equiv  
    \sum_{r=0}^{d} (-1)^{d+r}
    \beta_r \SigSubset{j+r}
    \mod \Iviz
  \end{equation*}
  for some coefficients $\beta_r \in \Z$.

  In order to determine these $\beta_r$, we count the number of
  power products of degree $j+r$ on the right-hand side of
  \eqref{eq:sumsumsumOfSigmaionej}. There are $\binom{n_\cH}{d+j+1}$
  possible choices for $S$, only the one particular fixed 
  partition~$S = T \cup W$ and $\binom{d+1}{r}$ possible choices for
  $U$ out of~$T$. Hence there are
  $\binom{n_\cH}{d+j+1}\binom{d+1}{r}$ power products of degree $j+r$
  appearing in \eqref{eq:sumsumsumOfSigmaionej}, all of which have the
  same sign. Due to the fact that $\SigSubset{j+r}$ contains
  $\binom{n_\cH}{j+r}$ monomials, this implies that
  \begin{equation*}
    \beta_r = \binom{n_\cH}{d+j+1}\binom{d+1}{r} \bigg/ \binom{n_\cH}{j+r} .
    \end{equation*}
\end{proof}

As mentioned, the lemmata above provide ``reduction rules'' for some
quantities $\SigSubset{i}$ or products of such quantities. We now
derive explicit formulas for particular instances.

\begin{remark}
  \label{remark:rules:SigSubset:d1}
  Suppose we have $d=1$, i.e., our full set-up is
  $k_\cG = n_\cG \geq 1$ and $k_\cH = n_\cH-1 \geq 1$. Then, because of
  Lemma~\ref{lem:reduceSigTooLarge} (with $d=1$ and $j=0$) and
  Lemma~\ref{lem:ProdOfSumsAsSumOfSums} (with $i=j=1$,
  see also Remark~\ref{remark:ProdOfSumsAsSumOfSums:small-values}),
  and because $\SigSubset{0} = 1$,
  we have
  \begin{align*}
    \SigSubset{2}
    &\equiv -\frac{1}{2}n_\cH(n_\cH-1)\SigSubset{0} + (n_\cH-1)\SigSubset{1}
      = -\frac{1}{2}n_\cH(n_\cH-1) + (n_\cH-1)\SigSubset{1}
      \mod \Iviz \\    
    \SigSubset{1}^2
    &\equiv \SigSubset{1} + 2\SigSubset{2}
      \equiv (2n_\cH-1)\SigSubset{1} - n_\cH(n_\cH-1)
      \mod \Iviz.
  \end{align*}
\end{remark}

\begin{remark}
  \label{remark:rules:SigSubset:d2}
  Suppose we have $d=2$, i.e., our full set-up is
  $k_\cG = n_\cG \geq 1$ and $k_\cH = n_\cH-2 \geq 1$. Then, because of
  Lemma~\ref{lem:reduceSigTooLarge} and
  Lemma~\ref{lem:ProdOfSumsAsSumOfSums}
  (see also Remark~\ref{remark:ProdOfSumsAsSumOfSums:small-values})
  we have
    \begin{align*}
    \SigSubset{3} &\equiv 
    \frac{1}{3!\, (n_\cH-3)!}\bigl( 
    n_\cH!\, \SigSubset{0} - 3(n_\cH - 1)!\SigSubset{1}
    + 6(n_\cH - 2)!\, \SigSubset{2}
    \bigr) \\
    &\equiv 
    \tfrac{1}{6}n_\cH(n_\cH-1)(n_\cH-2)\SigSubset{0} 
    - \tfrac{1}{2}(n_\cH - 1)(n_\cH-2)\SigSubset{1} +
    (n_\cH - 2)\SigSubset{2}
    \mod \Iviz
    \intertext{and}  
    \SigSubset{4} &\equiv 
    \frac{1}{4!\, (n_\cH-4)!}\bigl( 
    (n_\cH-1)!\, \SigSubset{1} - 6(n_\cH - 2)!\, \SigSubset{2}
    + 18(n_\cH - 3)!\, \SigSubset{3}
    \bigr) \\ 
    &\equiv
    \frac{1}{4!\, (n_\cH-4)!}\big( 
    3n_\cH!\,\SigSubset{0} - 8(n_\cH - 1)!\, \SigSubset{1}
    + 12(n_\cH - 2)!\, \SigSubset{2}
    \big) \\ 
    &\equiv
    \tfrac{1}{8}n_\cH(n_\cH-1)(n_\cH-2)(n_\cH-3)\SigSubset{0}
    - \tfrac{1}{3}(n_\cH-1)(n_\cH-2)(n_\cH-3)\SigSubset{1} + \\
    &\qquad \qquad  \qquad \qquad
    \tfrac{1}{2}(n_\cH-2)(n_\cH-3)\SigSubset{2} \mod \Iviz\\
    \end{align*}
    as well as
    \begin{align*}
    \SigSubset{1}^2 &\equiv  \SigSubset{1} + 2\SigSubset{2} \mod \Iviz\\
    \SigSubset{2}\SigSubset{1} &\equiv
    2\SigSubset{2} + 3\SigSubset{3} \\
    &\equiv
    \tfrac{1}{2}n_\cH(n_\cH-1)(n_\cH-2)\SigSubset{0} 
    - \tfrac{3}{2}(n_\cH - 1)(n_\cH-2)\SigSubset{1} + \\
    &\qquad \qquad  \qquad \qquad
    (3n_\cH - 4)\SigSubset{2}
    \mod \Iviz \\
    \SigSubset{2}^2 &\equiv
    \SigSubset{2} + 6\SigSubset{3} + 6\SigSubset{4} \\ 
    &\equiv
    \tfrac{1}{4}(3n_\cH-5)n_\cH(n_\cH-1)(n_\cH-2)\SigSubset{0}
    - (2n_\cH-3)(n_\cH-1)(n_\cH-2)\SigSubset{1} + \\
    &\qquad \qquad  \qquad \qquad
    (1 + 3(n_\cH-1)(n_\cH-2))\SigSubset{2} \mod \Iviz.
    \end{align*}
    As usual, we have $\SigSubset{0} = 1$ everywhere.
\end{remark}

\begin{remark}\label{remark:sigma-calculus}
  Let us fix $d$, i.e., our full set-up is
  $k_\cG = n_\cG \geq 1$ and $k_\cH = n_\cH-d \geq 1$,
  and let us fix $g \in V(\cG)$.

  Then,
  more systematically speaking, whenever $f$ is a finite $\K$-linear
  combination of terms of the form $\SigSubset{i}$ and
  $\SigSubset{i}\,\SigSubset{j}$ for non-negative integers $i$ and $j$,
  we can reduce~$f$ to a form
  \begin{equation*}
    f \equiv \sum_{i=0}^d \phi_i \SigSubset{i}
    \mod \Iviz
  \end{equation*}
  for efficiently computable~$\phi_i \in \K$.

  The idea is to use Lemma~\ref{lem:ProdOfSumsAsSumOfSums} for
  $\SigSubset{i}\,\SigSubset{j}$ in order to get rid of these
  products and replace them by terms of the form $\SigSubset{i}$.
  After this step, one can repeatedly use
  Lemma~\ref{lem:prodxghj_kGeqnG_kHeqnHMinusi} in order to replace all
  $\SigSubset{i}$ for $i > d$ by linear combinations of $\SigSubset{i}$
  with $i \le d$. All these operations are efficient; the coefficients of
  the individual steps are given directly in
  Lemmata~\ref{lem:ProdOfSumsAsSumOfSums}
  and~\ref{lem:prodxghj_kGeqnG_kHeqnHMinusi}.
\end{remark}

Finally let us mention that an implementation of the arithmetic described 
in the previous
remark, for example with SageMath~\cite{sagemath}, is handy:
It makes it easily possible to verify the results of
Remarks~\ref{remark:rules:SigSubset:d1}
and~\ref{remark:rules:SigSubset:d2}.

This completes the section on our auxiliary results which we need in
the following to prove our certificates.

\subsection{Certificates for \texorpdfstring{$k_\cG = n_\cG$}{kG=nG} and \texorpdfstring{$k_\cH = n_\cH$}{kH=nH}}
\label{sec:sosCertificaten_kGeqnG_kHeqnH_easy}
The easiest and almost trivial case is the one with $k_\cG = n_\cG$ and 
$k_\cH = n_\cH$, so $d = 0$.
We get the following certificate and
therefore have proven with our method that
Vizing's conjecture holds in this case.

\begin{theorem}\label{thm:sosCertificaten_kGeqnG_kHeqnH_easy}
    For $k_\cG = n_\cG \geq 1$ and $k_\cH = n_\cH \geq 1$,
    Vizing's conjecture is true as the polynomials
    \begin{equation*}
    s_{g} = 0
    \quad\text{for $g \in V(\cG)$}
    \end{equation*}
    are a $0$-sos certificate of $\fviz$.
\end{theorem}
Note that we can simplify this $0$-sos certificate of
Theorem~\ref{thm:sosCertificaten_kGeqnG_kHeqnH_easy} to an empty sum
using no polynomial, but we give the formulation of
Theorem~\ref{thm:sosCertificaten_kGeqnG_kHeqnH_easy} to highlight the
similarity to the other certificates we will present in this section.

\begin{proof}[Proof of Theorem~\ref{thm:sosCertificaten_kGeqnG_kHeqnH_easy}]
    We have $x_{gh} \equiv 1 \mod \Iviz$ for all $g \in V(\cG)$ and 
    $h\in V(\cH)$ as already mentioned in Remark~\ref{rem:Prod_of_xgh} 
    due to Lemma~\ref{lem:prodxghj_kGeqnG_kHeqnHMinusi}.
    Hence we obtain
    \begin{align*}
    \fviz &= -k_\cG k_\cH + \sum_{g \in V(\cG)} \sum_{h \in V(\cH) } x_{gh}
    \equiv -k_\cG k_\cH + n_\cG n_\cH 
    = 0  
    = \sum_{g \in V(\cG)} s_g^2 \mod \Iviz,
    \end{align*}
    so the $s_g$ form indeed a $0$-sos certificate for $\fviz$. 
\end{proof}

Note that the certificate of
Theorem~\ref{thm:sosCertificaten_kGeqnG_kHeqnH_easy} has the
lowest degree possible.

\subsection{Certificates for \texorpdfstring{$k_\cG = n_\cG$}{kG=nG} and \texorpdfstring{$k_\cH = n_\cH-1$}{kH=nH-1}}
\label{sec:sosCertificaten_kGeqnG_kHeqnHminus1_easy}
The easiest non-trivial case is the one with $k_\cG = n_\cG$ and 
$k_\cH = n_\cH-1$, so $d=1$. Using
the above machinery (explained in the methodology Section~\ref{sec:methodology})
we first found a complicated sum-of-squares certificate, 
which is presented in Appendix~\ref{sec:firstFoundMoreComplexCertificates_d1}.
We were eventually able to transform this complicated certificate to the
following much easier
certificate and 
therefore have proven with our method that Vizing's conjecture holds in this 
case.

\begin{theorem}\label{thm:sosCertificaten_kGeqnG_kHeqnHminus1_easy}
  For $k_\cG = n_\cG \geq 1$ and $k_\cH = n_\cH-1 \geq 1$,
  Vizing's conjecture is true as the polynomials
    \begin{equation*}
      s_{g} = \biggl( \sum_{h \in V(\cH)}  x_{gh} \biggr) - (n_\cH-1)
      \quad\text{for $g \in V(\cG)$}
    \end{equation*}
    are a $1$-sos certificate of $\fviz$.
\end{theorem}

\begin{proof}
  The polynomials $s_g$ can alternatively be written as
  $s_g = \SigSubset{1} - (n_\cH-1)$. Using 
  Remark~\ref{remark:rules:SigSubset:d1} yields
  \begin{align*}
    s_g^2 
    &= (\SigSubset{1} - (n_\cH-1))^2 
    = \SigSubset{1}^2 - 2(n_\cH-1)\SigSubset{1} + (n_\cH-1)^2 \\
    &\equiv (2n_\cH-1)\SigSubset{1} - n_\cH(n_\cH-1) - 2(n_\cH-1)\SigSubset{1} + (n_\cH-1)^2 \\
    &= \SigSubset{1} -(n_\cH - 1)
      \mod \Iviz.
  \end{align*}
  Consequently, this evaluates to
    \begin{equation*}
    \sum_{g \in V(\cG)} s_g^2 
    = \sum_{g \in V(\cG)} \bigl( \SigSubset{1} -(n_\cH - 1) \bigr)
    = -n_\cG(n_\cH - 1) + \sum_{g \in V(\cG)} \SigSubset{1} 
    = \fviz
      \mod \Iviz,
    \end{equation*}
    so the $s_g$ form indeed a $1$-sos certificate for $\fviz$. 
\end{proof}

Note that the certificate of
Theorem~\ref{thm:sosCertificaten_kGeqnG_kHeqnHminus1_easy} has the
lowest positive degree possible and furthermore only uses very particular
monomials of degree at most 1.

\subsection{Certificates for \texorpdfstring{$k_\cG = n_\cG$}{kG=nG} and \texorpdfstring{$k_\cH = n_\cH-2$}{kH=nH-2}}
\label{sec:sosCertificaten_kGeqnG_kHeqnHminus2_easy}

The next slightly more difficult case is the one for $k_\cG = n_\cG$ and 
$k_\cH = n_\cH-2$, so $d=2$.
Also in this case we first found a more complicated certificate (see
Appendix~\ref{sec:firstFoundMoreComplexCertificates_d2}) which we were able
to transform to the following simple certificate. 

\begin{theorem}\label{thm:sosCertificaten_kGeqnG_kHeqnHminus2_easy}
  For $k_\cG = n_\cG \geq 1$ and $k_\cH = n_\cH-2 \geq 1$,
  Vizing's conjecture is true as the polynomials
    \begin{equation*}
      s_{g} = \alpha
      + \beta \biggl( \sum_{h \in V(\cH)}   x_{gh} \biggr)
      + \gamma \biggl( \sum_{\set{h, h'} \subseteq V(\cH)} x_{gh}x_{gh'}\biggr)
      \quad\text{for $g \in V(\cG)$},
    \end{equation*}
    where
    \begin{align*}
      \alpha &= (n_\cH - 2)\bigl(n_\cH + \tfrac{1}{2}(n_\cH - 1)\sqrt{2}\bigr), \\
      \beta &= -\bigl( (2n_\cH-3) + (n_\cH - 2)\sqrt{2}\bigr), \\
      \gamma &= 2 + \sqrt{2},
    \end{align*}
    are a $2$-sos certificate of $\fviz$.
  \end{theorem}

\begin{remark}
  We want to point out that 
  Theorem~\ref{thm:sosCertificaten_kGeqnG_kHeqnHminus2_easy} is true
  whenever $\alpha$, $\beta$, $\gamma$ are solutions to the system of equations 
  \begin{subequations}
    \label{eq:systemForkHnHMinus2}
    \begin{align}
      -(n_\cH-2) &= \alpha^2 + \frac{1}{4}n_\cH(n_\cH-1)(n_\cH-2)(3n_\cH-5)\gamma^2 \hspace{15pt} \nonumber\\
      &\phantom{= \alpha^2}\; + n_\cH(n_\cH-1)(n_\cH-2)\beta\gamma,\\
        1 &= \beta^2 + 2\alpha\beta - (n_\cH-1)(n_\cH-2)(2n_\cH-3)\gamma^2 \nonumber\\
      &\phantom{= \beta^2}\; - 3(n_\cH-1)(n_\cH-2)\beta\gamma,\\
        0 &= 2\beta^2 + 2\alpha\gamma + (1+3(n_\cH-1)(n_\cH-2))\gamma^2 \nonumber\\
      &\phantom{= 2\beta^2}\; + 2(3n_\cH-4)\beta\gamma,
    \end{align}
  \end{subequations} 
  and that in Theorem~\ref{thm:sosCertificaten_kGeqnG_kHeqnHminus2_easy}
  one particular easy solution is stated.
\end{remark}

\begin{proof}[Proof of Theorem~\ref{thm:sosCertificaten_kGeqnG_kHeqnHminus2_easy}]
  The polynomials $s_g$ can alternatively be written as
  $s_g = \alpha\SigSubset{0} + \beta\SigSubset{1} +\gamma\SigSubset{2}$.
  (Note that $\SigSubset{0} = 1$.)
  Using Remark~\ref{remark:rules:SigSubset:d1} yields
  \begin{align*}
    s_g^2 
    &= (\alpha + \beta\SigSubset{1} +\gamma\SigSubset{2})^2 \\
    &= \alpha^2 + \beta^2\SigSubset{1}^2 + \gamma^2\SigSubset{2}^2
    + 2\alpha\beta\SigSubset{1} + 2\alpha\gamma\SigSubset{2} + 2\beta\gamma\SigSubset{1}\SigSubset{2} \\
    &\equiv
    \alpha^2 + 
    \beta^2(\SigSubset{1} + 2\SigSubset{2}) +
    2\alpha\beta\SigSubset{1} + 
    2\alpha\gamma\SigSubset{2}  \\
    &\phantom{\equiv \alpha^2}\;+ \gamma^2
    \bigl( 
    \tfrac{1}{4}(3n_\cH-5)n_\cH(n_\cH-1)(n_\cH-2)\SigSubset{0}
    - (2n_\cH-3)(n_\cH-1)(n_\cH-2)\SigSubset{1} \\
    &\hspace*{5em} +
    (1 + 3(n_\cH-1)(n_\cH-2))\SigSubset{2}    
    \bigr) \\
    &\phantom{\equiv \alpha^2}\; + 2\beta\gamma
    \bigl(   \tfrac{1}{2}n_\cH(n_\cH-1)(n_\cH-2)\SigSubset{0} 
    - \tfrac{3}{2}(n_\cH - 1)(n_\cH-2)\SigSubset{1}\\
    &\hspace*{5.5em} +
    (3n_\cH - 4)\SigSubset{2}
    \bigr) \mod \Iviz
  \end{align*}
  and consequently, we evaluate to
    \begin{align*}
     \sum_{g \in V(\cG)} s_g^2 
    &\equiv
    \bigl( 
    \alpha^2 +
    \gamma^2\tfrac{1}{4}(3n_\cH-5)n_\cH(n_\cH-1)(n_\cH-2)\\
    &\phantom{\equiv(\alpha^2}\hspace*{0.35em} +
    \beta\gamma n_\cH(n_\cH-1)(n_\cH-2)
    \bigr) \sum_{g \in V(\cG)} \SigSubset{0} \\
    &\phantom{\equiv}\;\, + \bigl( 
    \beta^2 + 
    2\alpha\beta -
    \gamma^2(2n_\cH-3)(n_\cH-1)(n_\cH-2) \\
    &\phantom{\equiv + (\beta^2}\hspace*{0.665em} -
    3\beta\gamma (n_\cH - 1)(n_\cH-2)
    \bigr) \sum_{g \in V(\cG)} \SigSubset{1} \\
    &\phantom{\equiv}\;\, + \bigl( 
    2\beta^2  +
    2\alpha\gamma +
    \gamma^2(1 + 3(n_\cH-1)(n_\cH-2)) \\
    &\phantom{\equiv + (2\beta^2}\hspace*{0.665em} +
    2\beta\gamma (3n_\cH - 4)
    \bigr) \sum_{g \in V(\cG)} \SigSubset{2}
      \mod \Iviz.
    \end{align*}
    Due to the particular values of $\alpha$, $\beta$ and $\gamma$ this simplifies to
    \begin{align*}
      \sum_{g \in V(\cG)} s_g^2
      &\equiv -(n_\cH-2)\sum_{g \in V(\cG)} \SigSubset{0} + \sum_{g \in V(\cG)} \SigSubset{1} \\
      &= -n_\cG k_\cH + \sum_{g \in V(\cG)} \SigSubset{1}
      = \fviz
      \mod \Iviz.
    \end{align*}
  \end{proof}




Note that for all computationally considered instances of the form $k_{\cG}=n_{\cG}$ and $k_{\cH}=n_{\cH}-2$, the SDP
for $\ell=1$ was infeasible, so for all of those instances there
seems to be no $1$-sos certificate and one really needs monomials of degree
2 in the $s_i$ in order to obtain a certificate.    
Nevertheless, degree 2 is still very low. Furthermore also in this sum-of-squares
certificate only very particular monomials are used; it can be considered
sparse therefore. This is confirmed by the following example.

\begin{example}
If we consider the case $k_\cG= n_\cG = 4$, $n_\cH = 5$ and $k_\cH=3$, there are $432$ monomials of degree at most 2 but the certificate of Theorem~\ref{thm:sosCertificaten_kGeqnG_kHeqnHminus2_easy} uses only $61$ of them.
\end{example}

\subsection{Computational Certificates for \texorpdfstring{$k_\cG = 
n_\cG$}{kG=nG} and \texorpdfstring{$k_\cH = n_\cH-d$}{kH=nH-d}}
\label{sec:sosCertificaten_kGeqnG_kHeqnHMinusd}

When taking a closer look at the certificates in
Theorem~\ref{thm:sosCertificaten_kGeqnG_kHeqnH_easy}, 
Theorem~\ref{thm:sosCertificaten_kGeqnG_kHeqnHminus1_easy} and
Theorem~\ref{thm:sosCertificaten_kGeqnG_kHeqnHminus2_easy}, one can
guess a structure from the certificates found so far.  In
particular there seems to be a $d$-sos certificate
for the case $k_\cG = n_\cG$ and
$k_\cH = n_\cH-d$.
Hence, at this point, we can formulate a conjecture which
intuitively seems to be the ``correct'' generalization.

\begin{conjecture}\label{con:sosCertificaten_kGeqnG_kHeqnHminusi_easy}
  For $k_\cG = n_\cG \ge 1$ and $k_\cH = n_\cH-d \ge 1$ with $d\geq 0$,
  Vizing's conjecture is true as the polynomials
    \begin{equation*}
    s_{g} = \sum_{i=0}^{d} \alpha_i \biggl(\sum_{\substack{S \subseteq 
    V(\cH)\\\card{S}=i}}\;\prod_{h \in S} x_{gh} \biggr)
    \quad\text{for $g \in V(\cG)$},
    \end{equation*}
    where $\alpha_i$ are the solutions to a certain system of polynomial 
    equations,
    are a $d$-sos certificate of $\fviz$.
\end{conjecture}

Moreover, the proofs of
Theorems~\ref{thm:sosCertificaten_kGeqnG_kHeqnHminus1_easy}
and~\ref{thm:sosCertificaten_kGeqnG_kHeqnHminus2_easy} give rise to an
algorithmic approach for finding certificates. We formulate this as
the following proposition.

\begin{proposition}\label{proposition:compute-certificate-kGeqnG_kHeqnHminusd}
  Let $d$ be a non-negative integer. Then there is an algorithm that
  either finds a certificate of the form as given in
  Conjecture~\ref{con:sosCertificaten_kGeqnG_kHeqnHminusi_easy} or
  outputs that there is no certificate of that form.
\end{proposition}

\begin{proof}
  We first describe our algorithm by following the proofs of
  Theorems~\ref{thm:sosCertificaten_kGeqnG_kHeqnHminus1_easy}
  and~\ref{thm:sosCertificaten_kGeqnG_kHeqnHminus2_easy}.

  Suppose $s_{g}$ is of the form as given in
  Conjecture~\ref{con:sosCertificaten_kGeqnG_kHeqnHminusi_easy}, so 
  \begin{equation*}
    s_{g} = \sum_{i=0}^{d} \alpha_i \SigSubset{i}
    \quad\text{for $g \in V(\cG)$},
  \end{equation*}
  where simply the definition of $\SigSubset{i}$ (see
  Definition~\ref{def:SigSubset_g_i}) was used.
  We now use binomial expansion for $s_{g}^2$. In
  the result there will be terms of the form
  $\SigSubset{i}\,\SigSubset{j}$ for $i$, $j \leq d$. We can now use
  the arithmetic described in Remark~\ref{remark:sigma-calculus} to end up with
  \begin{equation*}
    s_{g}^2 \equiv \sum_{i=0}^d \phi_i(\alpha)\, \SigSubset{i}
    \mod \Iviz
    \quad\text{for $g \in V(\cG)$}.
  \end{equation*}
  Here, $\phi_i(\alpha)$ is a polynomial in
  $\alpha = (\alpha_0, \dots, \alpha_d)$ for each $0 \leq i \leq d$.
  Note that all coefficients of this polynomial additionally depend on the
  variable~$n_\cH$.
  The formal summation over all ${g \in V(\cG)}$ is trivial; we obtain
  \begin{equation*}
    \sum_{g \in V(\cG)} s_{g}^2 \equiv \sum_{i=0}^d \phi_i(\alpha) \sum_{g \in V(\cG)}\SigSubset{i}
    \mod \Iviz.
  \end{equation*}
  In order to obtain a $d$-sos certificate, this has to be equal to
  \begin{equation*}
    \fviz = -(n_\cH-d)\sum_{g \in V(\cG)}\SigSubset{0}
    + \sum_{g \in V(\cG)}\SigSubset{1}.
  \end{equation*}
  Comparing the coefficients of
  $n_\cG = \sum_{g \in V(\cG)}\SigSubset{0}$ and the other
  $\sum_{g \in V(\cG)}\SigSubset{i}$ (for $1 \leq i \leq d$) yields
  the system of equations
  \begin{subequations}
   \label{eq:systemOfEquationsForConjecture}
  \begin{align}
    \phi_0(\alpha) &= -(n_\cH-d),\\
    \phi_1(\alpha) &= 1,\\
    \phi_i(\alpha) &= 0 \quad\text{for $2 \leq i \leq d$}.
  \end{align}
  \end{subequations}
  We want to point out that the existence of a real-valued solution
  $\alpha_0$, \dots, $\alpha_d$ (as functions in~$n_\cH$) is equivalent
  to the fact that to the $s_g$ being
  as in Conjecture~\ref{con:sosCertificaten_kGeqnG_kHeqnHminusi_easy} form
  a $d$-sos certificate.
  Therefore computing the variety associated to the system of
  equations~\eqref{eq:systemOfEquationsForConjecture}, i.e., finding all solutions of this system, 
  is the last step of an algorithm 
  that has the properties stated in
  Proposition~\ref{proposition:compute-certificate-kGeqnG_kHeqnHminusd},
  and the proof is completed.
\end{proof}

\begin{remark}
    The system of equations~\eqref{eq:systemOfEquationsForConjecture} 
    does not depend on $n_\cG$, but only on $d$ and $n_\cH$. 
    Hence, whenever we find a solution 
    to~\eqref{eq:systemOfEquationsForConjecture}---this might be
    for a fixed value of $n_\cH$ or parametrized in $n_\cH$---then
    this gives rise to a certificate for those values and
    \emph{all} possible values of $n_\cG=k_\cG$.
\end{remark}

Before we exploit the algorithm provided by
Proposition~\ref{proposition:compute-certificate-kGeqnG_kHeqnHminusd}
which finds a certificate of the 
form as given in Conjecture~\ref{con:sosCertificaten_kGeqnG_kHeqnHminusi_easy}, let us mention 
that it consists of two main steps:
The first step is to construct the system of 
equations~\eqref{eq:systemOfEquationsForConjecture} and the second is 
to 
find a solution to this system of equations.

Let us reconsider the proofs of 
Theorems~\ref{thm:sosCertificaten_kGeqnG_kHeqnHminus1_easy}
and~\ref{thm:sosCertificaten_kGeqnG_kHeqnHminus2_easy}. There, we already
have a particular certificate at hand, and we prove that it is in fact a certificate by 
performing essentially the first main step of the algorithm.
In fact the system of equations~\eqref{eq:systemForkHnHMinus2} corresponds 
to~\eqref{eq:systemOfEquationsForConjecture} for $d=2$ as the
variables
$(\alpha,\beta,\gamma)$ (of~\eqref{eq:systemForkHnHMinus2}) equal
$(\alpha_0,\alpha_1,\alpha_2)$ (of~\eqref{eq:systemOfEquationsForConjecture}). 
Even though the computations for proving the theorems above are tedious,
they are straight forward.

So, let us come back to the algorithm of 
Proposition~\ref{proposition:compute-certificate-kGeqnG_kHeqnHminusd}. For
finding a certificate for general $d \geq 3$ the situation is more difficult:
The computations get very
messy, so it seems infeasible to get the system of 
equations~\eqref{eq:systemOfEquationsForConjecture} in closed form depending on
the parameter~$d$. Moreover,
even for the case $d=2$ it is not
obvious that the system of equations~\eqref{eq:systemForkHnHMinus2}
even has a solution.
Still, we want to use 
Proposition~\ref{proposition:compute-certificate-kGeqnG_kHeqnHminusd}
for obtaining more certificates,
so let us consider the cases $d=1$ and $d=2$ once 
more, but this time with the help of SageMath~\cite{sagemath}.

Using the algorithm provided in the proof of
Proposition~\ref{proposition:compute-certificate-kGeqnG_kHeqnHminusd}
allows to reprove
Theorems~\ref{thm:sosCertificaten_kGeqnG_kHeqnHminus1_easy}
and~\ref{thm:sosCertificaten_kGeqnG_kHeqnHminus2_easy}
computationally with SageMath. It turns out that the variety of the system of 
equations~\eqref{eq:systemOfEquationsForConjecture}, whose points are the
solutions~$(\alpha_0, \dots, \alpha_d)$ 
of~\eqref{eq:systemOfEquationsForConjecture}, is of dimension~$1$ which
means that the dependency on~$n_\cH$ is the only dependency on a free
parameter. For $d=1$, the solution is essentially unique (except for the obvious
replacement of $s_g$ by $-s_g$). For $d=2$, in the solution presented in 
Theorem~\ref{thm:sosCertificaten_kGeqnG_kHeqnHminus2_easy} 
we can additionally replace each 
occurrence of $\sqrt{2}$ in any of $(\alpha,\beta,\gamma)$ by 
$-\sqrt{2}$ and obtain another solution.
In other words we can choose
the signs of $\pm\sqrt{1}$ and $\pm\sqrt{2}$.
This observation will be revisited in Remark~\ref{remark:question-alpha-d}.

In the same manner and with a lot of patience, we can let the
algorithm run for $d=3$ and get the result presented as 
Theorem~\ref{thm:sosCertificaten_kGeqnG_kHeqnHminus3} below. 
However, the computation can
be speeded up in the following way. This will allow to also
also cover the case $d=4$.

\begin{remark}\label{remark:compute-certificate-kGeqnG_kHeqnHminusd:general-nH}
  Suppose that the coefficients~$\alpha_0$, \dots, $\alpha_d$ are
  polynomials in~$n_\cH$ with degrees bounded by~$d$. (This is the case
  for Theorems~\ref{thm:sosCertificaten_kGeqnG_kHeqnHminus1_easy}
  and~\ref{thm:sosCertificaten_kGeqnG_kHeqnHminus2_easy}, so this
  assumption is reasonable.)  Then, by fixing a particular value
  for~$n_\cH$, the time of the computation of the $\alpha_0$, \dots, 
  $\alpha_d$ is now dramatically reduced. Doing this for $d+1$
  different values~$n_\cH$ allows to compute the
  coefficients~$\alpha_0$, \dots,  $\alpha_d$ as interpolation
  polynomials in~$n_\cH$.

  It should be noted that this interpolation trick is
  technically/computationally not as innocent as one might think: One
  has to carefully choose the values~$n_\cH$ in order to ``keep
  track'' of the branch of one particular solution,
  as the solution of~\eqref{eq:systemOfEquationsForConjecture} is not unique.
\end{remark}

By using the strategy explained in the previous remark, we are able to
show the following.

\begin{theorem}\label{thm:sosCertificaten_kGeqnG_kHeqnHminus3}
  For $k_\cG = n_\cG \geq 1$ and $k_\cH = n_\cH-3 \geq 1$,
  Vizing's conjecture is true as the polynomials
  \begin{equation*}
    s_{g} = \sum_{i=0}^3 \alpha_i \SigSubset{i} 
    \quad\text{for $g \in V(\cG)$},
  \end{equation*}
  where
  \begin{align*}
    \alpha_{0} &=
                 - \tfrac{1}{6} \, n_\cH^{3} {\left(\sqrt{3} + 3 \, \sqrt{2} + 3\right)}
                 + \tfrac{1}{2} \, n_\cH^{2} {\left(2 \, \sqrt{3} + 5 \, \sqrt{2} + 4\right)}\\
               &\phantom{=}\;
                 - \tfrac{1}{2} \, n_\cH {\left(\tfrac{11}{3} \sqrt{3} + 6\sqrt{2} + 3\right)}
                 + \sqrt{3}, \\
    \alpha_{1} &=
                 + \tfrac{1}{2} \, n_\cH^{2} {\left(\sqrt{3} + 3\sqrt{2} + 3\right)}
                 - \tfrac{1}{2} \, n_\cH {\left(5\sqrt{3} + 13\sqrt{2} + 11\right)}
                 + 3 \, \left(\sqrt{3} + 2\sqrt{2}\right)
                 + 4, \\
    \alpha_{2} &= - n_\cH {\left(\sqrt{3} + 3 \, \sqrt{2} + 3\right)}
                 + 3 \, \sqrt{3} + 8 \, \sqrt{2} + 7, \\
    \alpha_{3} &= \sqrt{3} + 3 \, \sqrt{2} + 3,
  \end{align*}
  are a $3$-sos certificate of $\fviz$.
\end{theorem}

\begin{proof}
  We apply the algorithm provided by
  Proposition~\ref{proposition:compute-certificate-kGeqnG_kHeqnHminusd}
  and
  Remark~\ref{remark:compute-certificate-kGeqnG_kHeqnHminusd:general-nH} and
  the claimed result follows.
  In particular we use SageMath~\cite{sagemath} in order to 
  construct the system of equations~\eqref{eq:systemOfEquationsForConjecture}
  and to obtain a solution of it.
\end{proof}

\begin{theorem}\label{thm:sosCertificaten_kGeqnG_kHeqnHminus4}
  For $k_\cG = n_\cG \geq 1$ and $k_\cH = n_\cH-4 \geq 1$,
  Vizing's conjecture is true as the polynomials
  \begin{equation*}
    s_{g} = \sum_{i=0}^4 \alpha_i \SigSubset{i} 
    \quad\text{for $g \in V(\cG)$},
  \end{equation*}
  where
  \begin{align*}
    \alpha_{0} &=
                 \tfrac{1}{12} \, n_\cH^{4} {\left(2 \, \sqrt{3} + 3 \, \sqrt{2} + 1\right)}
                 - \tfrac{1}{6} \, n_\cH^{3} {\left(9 \, \sqrt{3} + 12 \, \sqrt{2} + 2\right)} \\
               &\phantom{=}\;
                 + \tfrac{1}{12} \, n_\cH^{2} {\left(52 \, \sqrt{3} + 57 \, \sqrt{2} - 7\right)}
                 - \tfrac{1}{6} \, n_\cH {\left(24 \, \sqrt{3} + 18 \, \sqrt{2} - 17\right)} - 2, \\
    \alpha_{1} &=
                 -\tfrac{1}{3} \, n_\cH^{3} {\left(2 \, \sqrt{3} + 3 \, \sqrt{2} + 1\right)}
                 + \tfrac{1}{2} \, n_\cH^{2} {\left(11 \, \sqrt{3} + 15 \, \sqrt{2} + 3\right)} \\
               &\phantom{=}\;
                 - \tfrac{1}{6} \, n_\cH {\left(83 \, \sqrt{3} + 99 \, \sqrt{2} + 1\right)}
                 + 10 \, \sqrt{3} + 10 \, \sqrt{2} - 3, \\
    \alpha_{2} &=
                 n_\cH^{2} {\left(2 \, \sqrt{3} + 3 \, \sqrt{2} + 1\right)}
                 - n_\cH {\left(13 \, \sqrt{3} + 18 \, \sqrt{2} + 4\right)}
                 + 5 \left(4\sqrt{3} + 5\sqrt{2}\right)  + 2, \\
    \alpha_{3} &=
                 -2 \, n_\cH {\left(2 \, \sqrt{3} + 3 \, \sqrt{2} + 1\right)}
                 + 15 \, \sqrt{3} + 21 \, \sqrt{2} + 5, \\
    \alpha_{4} &=
                 4 \, \sqrt{3} + 6 \, \sqrt{2} + 2,
  \end{align*}
  are a $4$-sos certificate of $\fviz$.
\end{theorem}

\begin{proof}
  We again use SageMath~\cite{sagemath} and apply the algorithm provided by
  Proposition~\ref{proposition:compute-certificate-kGeqnG_kHeqnHminusd}
  and
  Remark~\ref{remark:compute-certificate-kGeqnG_kHeqnHminusd:general-nH}
  to obtain the claimed certificate.
\end{proof}

It should be noted once more that once having algorithmically proven
Theorems~\ref{thm:sosCertificaten_kGeqnG_kHeqnHminus3}
and~\ref{thm:sosCertificaten_kGeqnG_kHeqnHminus4}, verifying that
those results indeed form a certificate---again this can be done
computationally---is much easier.

\begin{remark}\label{remark:question-alpha-d}
  Let us consider the set-up and certificate as presented in
  Conjecture~\ref{con:sosCertificaten_kGeqnG_kHeqnHminusi_easy}
  again. In particular, let us have a look at the
  coefficient~$\alpha_d$ for various~$d$. By using the certificates
  obtained in this Section~\ref{sec:proof_of_kG=nG_and_kH=nH-1}, we
  may rewrite this coefficient as
  \begin{align*}
  d&=1\colon & \alpha_{1} &= \sqrt{1}, \\
  d&=2\colon & \alpha_{2} &= \sqrt{2} + (1+1) \, \sqrt{1}, \\
  d&=3\colon & \alpha_{3} &= \sqrt{3} + (1+2) \, \sqrt{2} + (1+1+1) \, \sqrt{1}, \\
  d&=4\colon & \alpha_{4} &= \sqrt{4} + (1+3) \, \sqrt{3} + (1+2+3) \, \sqrt{2}
               + (1+1+1+1) \, \sqrt{1}.
  \intertext{We therefore ask the following question: Is it true that}
  d&=5\colon & \alpha_{5} &= \sqrt{5} + (1+4) \, \sqrt{4} + (1+3+5) \, \sqrt{3} \\
               &&&\phantom{=}\;
               + (1+2+3+4) \, \sqrt{2} + (1+1+1+1+1) \, \sqrt{1}, \\
  d&=6\colon & \alpha_{6} &= \sqrt{6} + (1+5) \, \sqrt{5} + (1+4+7) \, \sqrt{4}
               + (1+3+5+7) \, \sqrt{3} \\
               &&&\phantom{=}\;
               + (1+2+3+4+5) \, \sqrt{2}
               + (1+1+1+1+1) \, \sqrt{1}
\end{align*}
and more generally for given $d$ that
\begin{equation}\label{eq:alphad}
  \alpha_d = \sum_{i=0}^{d-1} \Bigl(\;\sum_{j=0}^{d-i-1} (1+ij)\Bigr) \sqrt{i+1}
\end{equation}
is a choice for~$\alpha_d$ in a certificate for
Conjecture~\ref{con:sosCertificaten_kGeqnG_kHeqnHminusi_easy}? If so,
are \emph{all} possible certificates given by choosing a sign for each
square root $\pm\sqrt{i+1}$ in~\eqref{eq:alphad} (including the signs of expressions like
$\sqrt{1}=1$ and $\sqrt{4}=2$, i.e., $2^d$ different solutions)?
The latter turned out to be true for $d \in \set{0,1,2,3,4}$ by our computations.
\end{remark}

To summarize, in this section we have obtained certificates for the cases
with $k_\cG = n_\cG \ge 1$ and
$k_\cH=n_\cH - d$ for $d \in \set{0,1,2,3,4}$ by our method. For $d \in \set{0,1,2}$
we have proven these results by hand, for $d\in \set{3,4}$ we have proven them 
computationally.
We will continue to prove the correctness of certain certificates in the next 
section.

\section{Exact Certificates for \texorpdfstring{$k_\cG = n_\cG-1$}{kG=nG-1} and \texorpdfstring{$k_\cH=n_\cH - 1$}{kH=nH-1} with \texorpdfstring{$n_\cH\in \set{2,3}$}{nH=2 or nH=3} }
\label{sec:certificates_kGIsnGMinus1_kHEqnHMinus1}

In this section we will finally prove
Theorem~\ref{thm:sosCertificate_kGEqnGMinus1_kH2nH3} and therefore
obtain a certificate for the case $k_\cG = n_\cG-1 \ge 1$, $n_\cH=3$
and $k_\cH=2$. As a byproduct we will also obtain a certificate for the
case $k_\cG = n_\cG-1 \ge 1$, $n_\cH=2$ and $k_\cH=1$. Towards that
end we will use some of the results of
Section~\ref{sec:sigma-calculus} and derive further
results of a similar nature.

\subsection{Auxiliary Results}
\label{sec:auxiliaryResults_kGIsnGMinus1_nH3_kH2}
We start with the following lemma.
\begin{lemma}\label{lem:eggPrimeZero_kGeqnGMinus1}
    Let $k_\cG = n_\cG - 1 \ge 1$. Then $e_{gg'} \in I_\cG \subseteq \Iviz$
    holds for all $\set{g,g'} \subseteq D_\cG$.
\end{lemma}

This lemma is equivalent to $e_{gg'} \equiv 0 \mod I_\cG$ 
and $e_{gg'} \equiv 0 \mod \Iviz$ for
all $\set{g,g'} \subseteq D_\cG$. 

Lemma~\ref{lem:eggPrimeZero_kGeqnGMinus1} is plausible 
from a graph theoretical point of view. 
Indeed, due to Theorem~\ref{thm:BijectionVarietyGraphsG} the points in the 
variety of 
$I_\cG$ are in bijection to the graphs in $\cG$, which are the graphs on 
$n_\cG$ 
vertices with domination number $k_\cG=n_\cG-1$ and a minimum dominating set 
$D_\cG$.
Clearly in such graphs there are no edges 
between any two vertices of $D_\cG$, 
because if there would be such an edge, the domination number would decrease.
Hence, for each point in the variety of $\Iviz$, we have that the component
$e^\ast_{gg'}$ is zero for every $\set{g,g'} \subseteq D_\cG$.

\begin{proof}[Proof of Lemma~\ref{lem:eggPrimeZero_kGeqnGMinus1}]
  For $k_\cG = n_\cG - 1 = 1$ there is no $\set{g,g'} \subseteq D_\cG$,
  so there is nothing to prove.

    Let $\set{g,g'} \subseteq D_\cG$.
    We apply
    Hilbert's Nullstellensatz on
    the polynomial $f = e_{gg'}$. We have $k_\cG = n_\cG - 1$,
    therefore $\card{V(\cG)\setminus D_\cG} = 1$, and so let
    $\set{\hat{g}} = V(\cG)\setminus D_\cG$. Then clearly
    $g \neq \hat{g}$ and $g' \neq \hat{g}$.
    
    We use Notation~\ref{notation:poly-vs-evaluated-star-world}.
    Let $z^\ast \in \variety{I_\cG}$, so $z^\ast$ is a
    common zero of \eqref{eq:eij-fix}, \eqref{eq:domset-fix} and
    \eqref{eq:kcover-fix}.
    We assume that $z^\ast$ is not a zero of $f = e_{gg'}$.

    Let us set
    $S = S_g = \set{\widetilde{g} \in D_\cG \colon \widetilde{g} \neq g}$.
    Then, due to \eqref{eq:kcover-fix} we have
    \begin{equation} \label{eq:prod0_g}
    \parentheses[\bigg]{\;\sum_{ \tilde{g} \in S_g}e^\ast_{\tilde{g}g}}
    \parentheses[\bigg]{\;\sum_{ \tilde{g} \in S_g}e^\ast_{\tilde{g}\hat{g}}} = 0
    \end{equation}
    and conclude that one of the two factors has to be zero.

    Due to \eqref{eq:eij-fix}, all $e^\ast_{\tilde{g}g}$ and $e^\ast_{\tilde{g}\hat{g}}$
    appearing in~\eqref{eq:prod0_g} are in $\set{0,1}$, and moreover
    $e^\ast_{gg'} \in \set{0,1}$.
    Then $e^\ast_{gg'} = 1$ (as it is assumed to be non-zero), and,
    because $g' \in S_g$,
    the first factor of \eqref{eq:prod0_g} is non-zero.
    Therefore the second factor of~\eqref{eq:prod0_g} must be zero
    and hence $e^\ast_{\tilde{g}\hat{g}} = 0$ for all $\tilde{g} \in S_g$.

    By symmetry (switching the roles of $g$ and $g'$), we obtain
    $e^\ast_{\tilde{g}\hat{g}} = 0$ for all $\tilde{g} \in S_{g'}$ and
    therefore get
    $e^\ast_{\tilde{g}\hat{g}} = 0$ for all $\tilde{g} \in S_g \cup S_{g'} = D_\cG$.
    Thus,
    \begin{align*}
    \prod_{\tilde{g} \in D_\cG} (1 - e^\ast_{\tilde{g}\hat{g}}) = 1,
    \end{align*}
    but due to~\eqref{eq:domset-fix} this product should be zero; a contradiction.
    Hence $z^\ast$ is also a zero
    of $f=e_{gg'}$, and Hilbert's Nullstellensatz
    (Theorem~\ref{thm:HilbertsNullstellensatz} and
    Remark~\ref{remark:HilbertsNullstellensatz}) implies that
    $f = e_{gg'} \in I_\cG$.
\end{proof}

In particular we will need the following consequence of Lemma~\ref{lem:eggPrimeZero_kGeqnGMinus1}.

\begin{corollary}\label{cor:prodOfeEqual0_kGeqnGMinus1}
    Let $k_\cG = n_\cG - 1 \ge 1$. 
    Then $e_{{g_1}{g_2}}e_{g_{3}g_{4}} \in I_\cG \subseteq \Iviz$
    holds for all $\set{g_1,g_2}$, $\set{g_3,g_4} \subseteq V(\cG)$ with $\set{g_1,g_2} \neq \set{g_3,g_4}$.
\end{corollary}

Note that also Corollary~\ref{cor:prodOfeEqual0_kGeqnGMinus1} can be explained 
from a graph theoretic point of view. 
To be precise there can be only one edge 
in a graph on $n_\cG$ vertices with domination number $n_\cG-1$,
because an additional edge would decrease the domination number. 
Therefore, for each point in the variety of $I_\cG$, the product of components
corresponding to two different edge variables
has always to be equal to $0$ due to 
Theorem~\ref{thm:BijectionVarietyGraphsG}.

\begin{proof}[Proof of Corollary~\ref{cor:prodOfeEqual0_kGeqnGMinus1}]
  If $\set{g_1,g_2} \subseteq D_\cG$ or $\set{g_3,g_4} \subseteq D_\cG$
  the result follows from Lemma~\ref{lem:eggPrimeZero_kGeqnGMinus1}
  because then $e_{g_1 g_2} \in I_\cG$ or $e_{g_3 g_4} \in I_\cG$.
  Hence we only have to consider the case
  $\set{g_1,g_2} \not \subseteq D_\cG$ and
  $\set{g_3,g_4} \not \subseteq D_\cG$. Let
  $\set{\hat{g}} = V(\cG)\setminus D_\cG$, then without loss of
  generality this case is equivalent to $g_1 = g_3 = \hat{g}$ and
  $g_2 \neq g_4$.
    
    We use Hilbert's Nullstellensatz like in 
    Lemma~\ref{lem:eggPrimeZero_kGeqnGMinus1} to prove the statement. Let 
    $z^\ast \in \variety{I_\cG}$ be a common zero of \eqref{eq:eij-fix}, 
    \eqref{eq:domset-fix} and
    \eqref{eq:kcover-fix}. If we can prove that $z^\ast$ is also a zero of $f 
    = e_{\hat{g}{g_2}}e_{\hat{g}g_{4}}$ we are done.
    
    For $S = V(\cG)\setminus \set{g_2,g_4}$, \eqref{eq:kcover-fix} implies
    \begin{align*}
    \parentheses[\bigg]{\;\sum_{ g \in S}e^\ast_{g{g_2}}}
    \parentheses[\bigg]{\;\sum_{ g \in S}e^\ast_{g{g_4}}} = 0,
    \end{align*}
    so one of these two factors has to be zero; without loss of
    generality (due to symmetry in $g_2$ and $g_4$),
    let us assume the first factor.
    As $e^\ast_{\tilde{g}g_2} \in \set{0,1}$ for all $\tilde{g} \in V(\cG)$
    by~\eqref{eq:eij-fix}, we then have in fact $e^\ast_{\tilde{g}g_2}=0$
    for all $\tilde{g} \in V(\cG)$. In particular we have $e^\ast_{\hat{g}{g_2}}=0$
    because $\hat{g} \in S$. This is what we wanted to show.
\end{proof}

We need Corollary~\ref{cor:prodOfeEqual0_kGeqnGMinus1} in order to prove the next result.

\begin{lemma}\label{lem:prodxghj_kGeqnGMinus1_kHeqnHMinusi}
    Let $k_\cG = n_\cG - 1 \ge 1$, $n_\cH \in \set{2,3}$ and $k_\cH = n_\cH - 1$.
    Then
    \begin{equation*}
    (1-x_{gh_1})(1-x_{gh_2})(1-x_{g'h_3})(1-x_{g'h_4}) \in \Iviz
    \end{equation*}
    for $\set{g,g'} \subseteq V(\cG)$ and
    for $\set{h_1,h_2}$, $\set{h_3,h_4} \subseteq V(\cH)$.
\end{lemma}

Note that it would again be possible to justify 
Lemma~\ref{lem:prodxghj_kGeqnGMinus1_kHeqnHMinusi} in terms of graph theory 
using 
Theorem~\ref{def:IGHsdp}.
It would need a case distinction for $n_\cH =2$ and $n_\cH = 3$ 
and several more case distinctions on whether $g$, $g' \in D_\cG$, whether
$h_1$, $h_2$, $h_3$, $h_4 \in D_\cH$ and on the cardinality of $\{h_1,h_2,h_3,h_4\}$.
We refrain from presenting the details here.

\begin{proof}[Proof of Lemma~\ref{lem:prodxghj_kGeqnGMinus1_kHeqnHMinusi}]
    First observe that without loss of generality we can assume that $g \in D_\cG$, 
    as $\card{D_\cG} = n_\cG - 1$ and therefore not both of $g$ and $g'$ 
    can be in $V(\cG)\setminus D_\cG$.
    For notational convenience let $\set{\hat{g}} =  V(\cG)\setminus D_\cG$,
    and note that $g'$ might or might not be equal to $\hat{g}$. 
    
    Next observe that without loss of generality $h_4 = h_1$, 
    because $n_\cH \in \set{2,3}$, and $h_1 \neq h_2$ and $h_3 \neq h_4$
    by assumption.
    We obtain that the
    sets $\set{h_1,h_2}$ and $\set{h_3,h_4}$ are both of cardinality~$2$
    and not disjoint. 
    
    In order to prove Lemma~\ref{lem:prodxghj_kGeqnGMinus1_kHeqnHMinusi} 
    we will use Hilbert's Nullstellensatz for 
    $f = (1-x_{gh_1})(1-x_{gh_2})(1-x_{g'h_1})(1-x_{g'h_3})$ 
    analogously as it has been done in the proofs of
    Lemma~\ref{lem:eggPrimeZero_kGeqnGMinus1} and
    Corollary~\ref{cor:prodOfeEqual0_kGeqnGMinus1}.
    Note that we use Notation~\ref{notation:poly-vs-evaluated-star-world-2}.
    Towards that end let $z^\ast \in \variety{\Iviz}$, i.e., $z^\ast$ is a 
    common zero of
    \eqref{eq:eij-fix}, \eqref{eq:domset-fix} and \eqref{eq:kcover-fix}
    for both $\cG$ and $\cH$ and of \eqref{eq:xgh} and
    \eqref{eq:xgh_domset}.
    Due to \eqref{eq:eij-fix} and \eqref{eq:xgh} all of $e^\ast_{gg'}$,
    $e^\ast_{hh'}$ and $x^\ast_{gh}$ are either $0$ or $1$.
    
    Assume that $z^\ast$ is not a zero of $f$, 
    then $x^\ast_{gh_1} = x^\ast_{gh_2} = x^\ast_{g'h_1} = x^\ast_{g'h_3} = 0$. 
    Furthermore, as $g \in D_\cG$ 
    we have $e^\ast_{g\tilde{g}} = 0$ for all $\tilde{g} \in D_\cG$
    by Lemma~\ref{lem:eggPrimeZero_kGeqnGMinus1}.
    
    Now we distinguish the two cases $n_\cH = 2$ and $n_\cH = 3$.
    For $n_\cH = 2$ the above condition~$e^\ast_{g\tilde{g}} = 0$ for all $\tilde{g} \in D_\cG$ together with 
    \eqref{eq:xgh_domset} for the vertices~$gh_1$ and $gh_2$ of the box graph class imply
    \begin{align*}
    1-e^\ast_{g\hat{g}}x^\ast_{\hat{g}h_1} = 0 
    \quad \text{ and } \quad 
    1-e^\ast_{g\hat{g}}x^\ast_{\hat{g}h_2} = 0, 
    \end{align*}
    so $e^\ast_{g\hat{g}} = x^\ast_{\hat{g}h_1} = x^\ast_{\hat{g}h_2} = 1$ holds. 
    Note that for $n_\cH = 2$ we have $\set{h_1,h_2} = \set{h_3,h_4}$, 
    so $\hat{g} \neq g'$ and hence $g' \in D_\cG$ because $x^\ast_{g'h_1} = 0$.
    Then~\eqref{eq:xgh_domset} 
    for $g'h_1$ and $g'h_3$ yields 
    \begin{align*}
    1-e^\ast_{g'\hat{g}}x^\ast_{\hat{g}h_1} = 0 
    \quad \text{ and } \quad
    1-e^\ast_{g'\hat{g}}x^\ast_{\hat{g}h_3} = 0
    \end{align*}
    and hence $e^\ast_{g'\hat{g}} = x^\ast_{\hat{g}h_1} = x^\ast_{\hat{g}h_3} = 1$.
    But due to Corollary~\ref{cor:prodOfeEqual0_kGeqnGMinus1} and $e^\ast_{g\hat{g}} = 1$ 
    we have $e^\ast_{g'\hat{g}} = 0$, a contradiction.
    Hence the lemma holds for $n_\cH = 2$. 
    
    Next we consider the case $n_\cH = 3$. Here we let $\set{h} = V(\cH) \setminus \set{h_1,h_2}$ 
    and let $\set{h'} = V(\cH) \setminus \set{h_1,h_3}$.
    Together with \eqref{eq:xgh_domset} for the box graph class vertices~$gh_1$ and $gh_2$, 
    the above derived fact $e^\ast_{g\tilde{g}} = 0$ for all $\tilde{g} \in D_\cG$
    yields
    \begin{subequations}
    \begin{align}
    (1-e^\ast_{g\hat{g}}x^\ast_{\hat{g}h_1})(1-e^\ast_{h_1 h}x^\ast_{gh}) 
    &= 0 \quad \text{ and } \label{eq:gh1}\\
    (1-e^\ast_{g\hat{g}}x^\ast_{\hat{g}h_2})(1-e^\ast_{h_2 h}x^\ast_{gh}) 
    &= 0. \label{eq:gh2}
    \end{align}
    \end{subequations}
    If $e^\ast_{g\hat{g}} = 0$, 
    then \eqref{eq:gh1} and \eqref{eq:gh2} imply $x^\ast_{gh} = 1$ 
    and $e^\ast_{h_1 h} = e^\ast_{h_2 h} = 1$, 
    which is a contradiction to Corollary~\ref{cor:prodOfeEqual0_kGeqnGMinus1}. 
    So $e^\ast_{g\hat{g}} = 1$ holds. 
    Now we will distinguish the two cases $g' \neq \hat{g}$ and $g' = \hat{g}$.
    
    \emph{Case $g' \neq \hat{g}$.}
    We have $g' \in D_\cG$ and can deduce from~\eqref{eq:xgh_domset} 
    for $g'h_1$ and $g'h_3$ analogously as for $g$ that 
    \begin{subequations}
    \begin{align}
    (1-e^\ast_{g'\hat{g}}x^\ast_{\hat{g}h_1})(1-e^\ast_{h_1 h'}x^\ast_{g'h'}) 
    &= 0 \quad \text{ and } \label{eq:gPrimeh1} \\
    (1-e^\ast_{g'\hat{g}}x^\ast_{\hat{g}h_3})(1-e^\ast_{h_3 h'}x^\ast_{g'h'}) 
    &= 0. \label{eq:gPrimeh3}
    \end{align}
    \end{subequations}
    Due to Corollary~\ref{cor:prodOfeEqual0_kGeqnGMinus1} and $e^\ast_{g\hat{g}} = 1$ 
    we have $e^\ast_{g'\hat{g}} = 0$. 
    Therefore~\eqref{eq:gPrimeh1} and~\eqref{eq:gPrimeh3} imply 
    that $e^\ast_{h_1 h'} = e^\ast_{h_3 h'} = 1$, 
    which is a contradiction to Corollary~\ref{cor:prodOfeEqual0_kGeqnGMinus1}. 
    So in this case $z^\ast$ is also a zero of $f$ and hence $f \in \Iviz$
    holds because of Hilbert's Nullstellensatz.
    
    \emph{Case $g' = \hat{g}$.}
    Due to 
    Corollary~\ref{cor:prodOfeEqual0_kGeqnGMinus1} and $e^\ast_{g\hat{g}} = 1$ 
    we can deduce that
     $e^\ast_{g'\tilde{g}} = 0$ for all $\tilde{g} \in V(\cG)\setminus \set{g}$. 
     Therefore \eqref{eq:xgh_domset} for the vertices $g'h_1$ and $g'h_3$
     of the box graph class become
    \begin{subequations}
    \begin{align}
    (1-e^\ast_{g\hat{g}}x^\ast_{gh_1})(1-e^\ast_{h_1 h'}x^\ast_{\hat{g}h'}) 
    &= 0 \quad \text{ and } \label{eq:gPrimeh1_Case2} \\
    (1-e^\ast_{g\hat{g}}x^\ast_{gh_3})(1-e^\ast_{h_3 h'}x^\ast_{\hat{g}h'}) 
    &= 0. \label{eq:gPrimeh3_Case2}
    \end{align}
    \end{subequations}
    We have $x^\ast_{\hat{g}h_1} = x^\ast_{gh_1} = 0$, 
    so from \eqref{eq:gh1} and \eqref{eq:gPrimeh1_Case2} 
    it follows that $x^\ast_{gh} = x^\ast_{\hat{g}h'} = 1$
    and $e^\ast_{h_1 h} = e^\ast_{h_1 h'} = 1$.
    Corollary~\ref{cor:prodOfeEqual0_kGeqnGMinus1} applied on the graph class~$\cH$ implies $h = h'$ 
    and therefore also $h_2 = h_3$ holds. 
    Furthermore, this corollary also implies that $e^\ast_{h_2 h} = 0$, 
    and hence $x^\ast_{\hat{g}h_2} = 1$ because of \eqref{eq:gh2}. 
    But this is a contradiction because $x^\ast_{\hat{g}h_2} = x^\ast_{g' h_3} = 0$. 
    So also in this case $z^\ast$ is a zero of $f$ and therefore $f \in \Iviz$ holds.
\end{proof}

\begin{remark}\label{rem:productOf4For_kGisnGMinus1_nH3_kH2}
    In particular for $k_\cG = n_\cG - 1$, $n_\cH \in \set{2,3}$ and $k_\cH = n_\cH - 1$, 
    Lemma~\ref{lem:prodxghj_kGeqnGMinus1_kHeqnHMinusi} implies
    \begin{align*}   
    x_{gh_1}x_{gh_2}x_{g'h_3}x_{g'h_4} &\equiv  
     x_{gh_1}x_{gh_2}x_{g'h_3}
    +x_{gh_1}x_{gh_2}x_{g'h_4}
    +x_{gh_1}x_{g'h_3}x_{g'h_4}
    +x_{gh_2}x_{g'h_3}x_{g'h_4}\\
    &\qquad -x_{gh_1}x_{gh_2}
    -x_{gh_1}x_{g'h_3}
    -x_{gh_2}x_{g'h_3}\\
    &\qquad -x_{gh_1}x_{g'h_4}
    -x_{gh_2}x_{g'h_4}
    -x_{g'h_3}x_{g'h_4}\\
    &\qquad +x_{gh_1}
    +x_{gh_2}
    +x_{g'h_3}
    +x_{g'h_4}
    -1  \mod  \Iviz
    \end{align*}
    for all  $\set{g,g'} \subseteq V(\cG)$ and
    all $\set{h_1,h_2}$, $\set{h_3,h_4} \subseteq V(\cH)$.
\end{remark}

Next we will need some more polynomials in order to be able to cope with $\SigSubset{i}$ in a better way.

\begin{definition}\label{def:tau_i_j}
    Let $i$ and $j$ be two non-negative integers. We define
    \begin{align*}
    \tauij{i}{j} = 
    \sum_{g \in V(\cG)}\; \sum_{\substack{g' \in V(\cG) \\ g' \neq g}}\;
    \SigSubset{i} \SigPrimeSubset{j}.
    \end{align*}
\end{definition}
Observe that $\tauij{i}{j} = \tauij{j}{i}$ holds. 
As a next step we will use Lemma~\ref{lem:prodxghj_kGeqnGMinus1_kHeqnHMinusi} 
in order to determine $\tauij{2}{2}$.

\begin{lemma} \label{lem:compute_tau22}
Let $k_\cG = n_\cG - 1\ge 1$, $n_\cH \in \set{2,3}$ and $k_\cH = n_\cH - 1$.
Then    
    \begin{align*}
    \tauij{2}{2}
    &\equiv 
    2(n_\cH - 1)\tauij{2}{1} 
    - (n_\cH - 1)^2\tauij{1}{1}
    - n_\cH(n_\cH - 1)(n_\cG - 1)\sum_{g \in V(\cG)}\SigSubset{2}\\
    &\quad + n_\cH(n_\cH - 1)^2(n_\cG - 1)\sum_{g \in V(\cG)}\SigSubset{1}
    - \frac{1}{4}n_\cG(n_\cG-1)n_\cH^2(n_\cH-1)^2  \mod \Iviz.
    \end{align*}
\end{lemma}
\begin{proof}
    By definition 
    \begin{align*}
    \tauij{2}{2}
    &= \sum_{g \in V(\cG)}\; \sum_{\substack{g' \in V(\cG) \\ g' \neq g}}\;
    \SigSubset{2} \SigPrimeSubset{2} \\
    &= \sum_{g \in V(\cG)}\; \sum_{\substack{g' \in V(\cG) \\ g' \neq g}}\;
    \left(\sum_{\set{h_1,h_2} \subseteq V(\cH)} x_{gh_1}x_{gh_2}\right)
    \left(\sum_{\set{h_3,h_4} \subseteq V(\cH)} x_{g'h_3}x_{g'h_4}\right)
    \end{align*}
    holds. 
    By using Lemma~\ref{lem:prodxghj_kGeqnGMinus1_kHeqnHMinusi} 
    as stated in Remark~\ref{rem:productOf4For_kGisnGMinus1_nH3_kH2} we obtain
    \begin{equation}\label{eq:tau22_expanded}
    \begin{aligned}
    \tauij{2}{2}
    &\equiv \sum_{g \in V(\cG)}\; \sum_{\substack{g' \in V(\cG) \\ g' \neq g}}\;
    \sum_{\substack{\set{h_1,h_2} \subseteq V(\cH) \\ \set{h_3,h_4} \subseteq V(\cH)}}
     \Big(x_{gh_1}x_{gh_2}x_{g'h_3}
     +x_{gh_1}x_{gh_2}x_{g'h_4}\\
     &\hspace*{5.5cm}+x_{gh_1}x_{g'h_3}x_{g'h_4}
     +x_{gh_2}x_{g'h_3}x_{g'h_4}\\
     &\hspace*{5.5cm} -x_{gh_1}x_{gh_2}
     -x_{gh_1}x_{g'h_3}
     -x_{gh_2}x_{g'h_3}\\
     &\hspace*{5.5cm} -x_{gh_1}x_{g'h_4}
     -x_{gh_2}x_{g'h_4}
     -x_{g'h_3}x_{g'h_4}\\
     &\hspace*{5.5cm} +x_{gh_1}
     +x_{gh_2}
     +x_{g'h_3}
     +x_{g'h_4}\\
     &\hspace*{5.5cm}-1 
    \Big) \mod \Iviz.
    \end{aligned}
    \end{equation}
    We can further reformulate~\eqref{eq:tau22_expanded} by using the following argument.  
    The monomials $x_{gh_1}x_{gh_2}x_{g'h_3}$ and $x_{gh_1}x_{gh_2}x_{g'h_4}$
    in~\eqref{eq:tau22_expanded} are both of the form $x_{gh_1}x_{gh_2}x_{g'h}$ 
    for some $\set{h_1,h_2} \subseteq V(\cH)$ and some $h \in V(\cH)$. 
    Hence due to symmetry
    \begin{align}\label{eq:explain_tau22_1}
    \sum_{\substack{\set{h_1,h_2} \subseteq V(\cH) \\ \set{h_3,h_4} \subseteq V(\cH)}}
    x_{gh_1}x_{gh_2}x_{g'h_3} +x_{gh_1}x_{gh_2}x_{g'h_4}
    \end{align}
    can be written as
    \begin{align}\label{eq:explain_tau22_2}
    \delta\sum_{\substack{\set{h_1,h_2} \subseteq V(\cH)\\ h \in V(\cH)}} 
    x_{gh_1}x_{gh_2}x_{g'h}
    \end{align}
    for some $\delta \in \Z$. 
    In order to compute $\delta$ observe that 
    there are $2\binom{n_\cH}{2}^2$ monomials of the considered form in~\eqref{eq:explain_tau22_1} 
    and that there are $n_\cH\binom{n_\cH}{2}$ monomials  of the considered form in~\eqref{eq:explain_tau22_2}.
    Hence $\delta = 2\binom{n_\cH}{2}^2/(n_\cH\binom{n_\cH}{2})$. 
    Similar arguments for the other monomials of~\eqref{eq:tau22_expanded} yield
    \begin{align*}
    \tauij{2}{2}
    &\equiv \sum_{g \in V(\cG)}\; \sum_{\substack{g' \in V(\cG) \\ g' \neq g}}\;
    \Bigg(
    \frac{2\binom{n_\cH}{2}^2}{n_\cH\binom{n_\cH}{2}}
    \sum_{\substack{\set{h_1,h_2} \subseteq V(\cH)\\ h \in V(\cH)}} 
    x_{gh_1}x_{gh_2}x_{g'h} \\
    &\hspace*{3.5cm} + \frac{2\binom{n_\cH}{2}^2}{n_\cH\binom{n_\cH}{2}}
    \sum_{\substack{\set{h_3,h_4} \subseteq V(\cH)\\ h \in V(\cH)}} 
    x_{gh}x_{g'h_3}x_{g'h_4} \\
    &\hspace*{3.5cm} - \frac{\binom{n_\cH}{2}^2}{\binom{n_\cH}{2}}\SigSubset{2} 
    - \frac{4\binom{n_\cH}{2}^2}{n_\cH^2}
    \sum_{\substack{h \in V(\cH) \\ h' \in V(\cH)}} 
    x_{gh}x_{g'h'} 
    - \frac{\binom{n_\cH}{2}^2}{\binom{n_\cH}{2}}\SigPrimeSubset{2}\\ 
    &\hspace*{3.5cm} + \frac{2\binom{n_\cH}{2}^2}{n_\cH}\SigSubset{1}
    + \frac{2\binom{n_\cH}{2}^2}{n_\cH}\SigPrimeSubset{1}
    - \binom{n_\cH}{2}^2
    \Bigg) \mod \Iviz,
    \end{align*}
    which, using the definition of $\tauij{i}{j}$ and  $\tauij{1}{2} = \tauij{2}{1}$, can be simplified to
    \begin{align*}
    \tauij{2}{2}
    &\equiv 
    2(n_\cH - 1)\tauij{2}{1} 
    - (n_\cH - 1)^2\tauij{1}{1}
    - n_\cH(n_\cH - 1)(n_\cG - 1)\sum_{g \in V(\cG)}\SigSubset{2}\\
    &\quad + n_\cH(n_\cH - 1)^2(n_\cG - 1)\sum_{g \in V(\cG)}\SigSubset{1}
    - \frac{1}{4}n_\cG(n_\cG-1)n_\cH^2(n_\cH-1)^2  \mod \Iviz.
    \end{align*}
\end{proof}


This completes the collection of result that we need in this section.

\subsection{Certificates for \texorpdfstring{$k_\cG = n_\cG-1$}{kG=nG-1}  and \texorpdfstring{$k_\cH=n_\cH -1$}{kH=nH-1} with \texorpdfstring{$n_\cH \in \set{2,3}$}{nH=2 or nH=3}}
\label{sec:certificate_kGEqnGMinus1_nH3_kH2}

Now we are finally able to prove Theorem~\ref{thm:sosCertificate_kGEqnGMinus1_kH2nH3}, which 
provides a sum-of-squares certificate of degree~$2$ for $k_\cG = n_\cG - 1 \ge 1$, $n_\cH = 3$ 
and $k_\cH = 2$. In fact we will prove the existence of sum-of-squares certificates not only in 
this case, but also for the case $k_\cG = n_\cG - 1 \ge 1$, $n_\cH = 2$ and $k_\cH = 1$
in the following theorem.

\begin{theorem} \label{thm:sosCertificate_kGEqnGMinus1_kH1nH2}
    For $k_\cG = n_\cG - 1 \ge 1$, $n_\cH \in \set{2,3}$ and $k_\cH = n_\cH - 1$
    Vizing's conjecture is true as the polynomials
    \begin{align*}
    s_{0} &=  \alpha
    + \beta \biggl( \sum_{g \in V(\cG)} \sum_{h \in V(\cH)} x_{gh} \biggr)
    + \gamma \biggl(\sum_{g \in V(\cG)}  \sum_{\set{h, h'} \subseteq V(\cH)} x_{gh}x_{gh'}\biggr)   
    \intertext{and}
    s_{g} &=  \kappa \biggl(\sum_{h \in V(\cH)}  x_{gh}  \biggr)
    - \lambda \biggl( \sum_{\set{h, h'} \subseteq V(\cH)}  x_{gh}x_{gh'} \biggr)
    \quad\text{for $g \in V(\cG)$},
    \end{align*}
    where $\alpha = \sqrt{n_\cH - 1}(n_\cG-1)$, $\beta = -\sqrt{n_\cH - 1}/n_\cH$
    and $\gamma = 2/(n_\cH\sqrt{n_\cH - 1})$,
    are a sum-of-squares certificate with degree~$2$ of $\fviz$.
    
    In particular for $n_\cH = 2$ we have $\alpha = n_\cG-1$, $\beta = -1$, $\gamma = 1$, $\kappa = 0$ and $\lambda = 1$
    and for $n_\cH = 3$ we have $\alpha = \sqrt{2}(n_\cG-1)$, $\beta = -\frac{2}{3}\sqrt{2}$
    $\gamma = \frac{1}{3}\sqrt{2}$, $\kappa = \frac{1}{3}$ and $ \lambda= -\frac{2}{3}$.
\end{theorem}

\begin{proof}[Proof of Theorems~\ref{thm:sosCertificate_kGEqnGMinus1_kH2nH3}
  and~\ref{thm:sosCertificate_kGEqnGMinus1_kH1nH2}]    
In order to prove that the polynomials $s_0$ and $s_g$ for $g \in V(\cG)$ are a sum-of-squares certificate we have to show that $s_0^2 + \sum_{g \in V(\cG)}s_g^2 \equiv \fviz \mod \Iviz$. 
Towards that end we can rewrite the polynomials as $s_0 = \alpha + \beta \sum_{g \in V(\cG)}\SigSubset{1} + \gamma\sum_{g \in V(\cG)}\SigSubset{2}$ and $s_{g} = \kappa\SigSubset{1} + \lambda \SigSubset{2})$. This yields
\begin{align*}
s_0^2 + \sum_{g \in V(\cG)}s_g^2 
&= 
\left(\alpha + \beta \sum_{g \in V(\cG)}\SigSubset{1} + \gamma\sum_{g \in V(\cG)}\SigSubset{2}\right)^2 
+ \sum_{g \in V(\cG)} (\kappa\SigSubset{1} + \lambda\SigSubset{2})^2 \\
&=
\alpha^2 
+ 2\alpha\beta \sum_{g \in V(\cG)}\SigSubset{1} 
+ 2\alpha\gamma\sum_{g \in V(\cG)}\SigSubset{2}
+ \beta^2\bigg( \sum_{g \in V(\cG)}\SigSubset{1}\bigg)\bigg( \sum_{g' \in V(\cG)}\SigPrimeSubset{1}\bigg)\\
&\phantom{= \alpha^2}\hspace*{0.3em}
+ \gamma^2\bigg( \sum_{g \in V(\cG)}\SigSubset{2}\bigg)\bigg( \sum_{g' \in V(\cG)}\SigPrimeSubset{2}\bigg)
+ 2\beta\gamma \bigg( \sum_{g \in V(\cG)}\SigSubset{1}\bigg)\bigg( \sum_{g' \in V(\cG)}\SigPrimeSubset{2}\bigg) \\
&\phantom{= \alpha^2}\hspace*{0.3em}
+\sum_{g \in V(\cG)}\Big( \kappa^2\SigSubset{1}^2 + 2\kappa\lambda\SigSubset{1}\SigSubset{2} + \lambda^2\SigSubset{2}^2 \Big)
\end{align*}
and therefore
\begin{equation}\label{eq:sumofsquares_intermediate}
\begin{aligned}
s_0^2 + \sum_{g \in V(\cG)}s_g^2 &=
\alpha^2 
+\sum_{g \in V(\cG)} \sum_{\substack{g' \in V(\cG)\\g' \neq g}} \left(
\beta^2 \SigSubset{1}\SigPrimeSubset{1}
+\gamma^2 \SigSubset{2}\SigPrimeSubset{2}
+2\beta\gamma \SigSubset{1}\SigPrimeSubset{2}
\right) \\
&\phantom{= \alpha^2}\hspace*{0.275em}
+\bigg(\sum_{g \in V(\cG)} ( \beta^2 + \kappa^2)\SigSubset{1}^2 
+ (2\beta\gamma + 2\kappa\lambda)\SigSubset{1}\SigSubset{2} 
+ (\gamma^2 + \lambda^2)\SigSubset{2}^2 \\
&\hspace*{7.25em}
+2\alpha\beta\SigSubset{1} 
+2\alpha\gamma\SigSubset{2}
\bigg)
\end{aligned}
\end{equation}
holds.
By using Lemma~\ref{lem:compute_tau22} we obtain
\begin{align*}
\sum_{g \in V(\cG)} \sum_{\substack{g' \in V(\cG)\\g' \neq g}} &\left(
\beta^2 \SigSubset{1}\SigPrimeSubset{1}
+\gamma^2 \SigSubset{2}\SigPrimeSubset{2}
+2\beta\gamma \SigSubset{2}\SigPrimeSubset{1}
\right)\\
&=
\beta^2\tauij{1}{1} +\gamma^2\tauij{2}{2} +2\beta\gamma\tauij{1}{2}\\
&\equiv 
\left(2\beta\gamma +  2(n_\cH - 1)\gamma^2\right)\tauij{2}{1} 
+ \left( \beta^2 - (n_\cH - 1)^2\gamma^2\right)\tauij{1}{1}\\
&\phantom{\equiv}\;\,
- n_\cH(n_\cH - 1)(n_\cG - 1)\gamma^2\sum_{g \in V(\cG)}\SigSubset{2}\\
&\phantom{\equiv}\;\,
+ n_\cH(n_\cH - 1)^2(n_\cG - 1)\gamma^2 \sum_{g \in V(\cG)}\SigSubset{1}\\
&\phantom{\equiv}\;\,
- \tfrac{1}{4}n_\cG(n_\cG-1)n_\cH^2(n_\cH-1)^2\gamma^2  \mod \Iviz.
\end{align*}
Furthermore,
we rewrite $\SigSubset{1}^2$, $\SigSubset{2}\SigSubset{1}$ and $\SigSubset{2}^2$
as in Remark~\ref{remark:ProdOfSumsAsSumOfSums:small-values}, so
\begin{align*}
(\beta^2 + \kappa^2)\SigSubset{1}^2 
&+ (2\beta\gamma + 2\kappa\lambda)\SigSubset{1}\SigSubset{2} 
+ (\gamma^2 + \lambda^2)\SigSubset{2}^2 
+2\alpha\beta\SigSubset{1} 
+2\alpha\gamma\SigSubset{2}\\
&\equiv(\beta^2 + \kappa^2)(\SigSubset{1} + 2\SigSubset{2}) 
+ (2\beta\gamma + 2\kappa\lambda)(2\SigSubset{2} + 3\SigSubset{3}) \\
&\phantom{\equiv}\;\,
+ (\gamma^2 + \lambda^2)(\SigSubset{2} + 6\SigSubset{3} + 6\SigSubset{4})
+2\alpha\beta\SigSubset{1} 
+2\alpha\gamma\SigSubset{2}\\
&\equiv 
(\beta^2 + \kappa^2 +2\alpha\beta)\SigSubset{1} + 
(2\beta^2 + 2\kappa^2 + 4\beta\gamma + 4\kappa\lambda + \gamma^2 + \lambda^2 +2\alpha\gamma)\SigSubset{2}\\
&\phantom{\equiv}\;\, +
(6\beta\gamma + 6\kappa\lambda + 6\gamma^2 + 6\lambda^2)\SigSubset{3} +
(6\gamma^2 + 6\lambda^2)\SigSubset{4} \mod \Iviz.
\end{align*}
The previous two identities together with \eqref{eq:sumofsquares_intermediate} 
and the fact that $\SigSubset{4} = 0$ trivially holds because $n_\cH \in \set{2,3}$ yield
\begin{align*}
s_0^2 + \sum_{g \in V(\cG)}s_g^2 
&\equiv
\bigl(\alpha^2 - \tfrac{1}{4}n_\cG(n_\cG-1)n_\cH^2(n_\cH-1)^2\gamma^2 \bigr) 
 + 2\bigl(\beta\gamma + (n_\cH - 1)\gamma^2\bigr)\tauij{2}{1} \\
&\phantom{\equiv}\;\,
+ \bigl( \beta^2 - (n_\cH - 1)^2\gamma^2\bigr)\tauij{1}{1} \\
&\phantom{\equiv}\;\,
+\bigl(\beta^2 + \kappa^2 +2\alpha\beta + n_\cH(n_\cH - 1)^2(n_\cG - 1)\gamma^2\bigr)\sum_{g \in V(\cG)}\SigSubset{1} \\
&\phantom{\equiv}\;\,
+ 
\bigl(2\beta^2 + 2\kappa^2 + 4\beta\gamma + 4\kappa\lambda + \gamma^2 + \lambda^2 \\
&\hspace*{4.425em}
+2\alpha\gamma - n_\cH(n_\cH - 1)(n_\cG - 1)\gamma^2\bigr)\sum_{g \in V(\cG)}\SigSubset{2}\\
&\phantom{\equiv}\;\,
+
6(\beta\gamma + \kappa\lambda + \gamma^2 + \lambda^2)\sum_{g \in V(\cG)}\SigSubset{3}
 \mod \Iviz.
\end{align*}
In order to obtain a certificate $s_0^2 + \sum_{g \in V(\cG)}s_g^2 \equiv \fviz = -k_\cG k_\cH + \sum_{g \in V(\cG)} \SigSubset{1} \mod \Iviz$ has to hold.
For $n_\cH = 2$ we have $\SigSubset{3} = 0$, so in this case 
every solution to the system of equations
\begin{align*}
\alpha^2 - \frac{1}{4}n_\cG(n_\cG-1)n_\cH^2(n_\cH-1)^2\gamma^2 & = -(n_\cH - 1)(n_\cG-1),\\
\beta\gamma +  (n_\cH - 1)\gamma^2 & = 0,\\
\beta^2 - (n_\cH - 1)^2\gamma^2 & = 0,\\
\beta^2 + \kappa^2 +2\alpha\beta + n_\cH(n_\cH - 1)^2(n_\cG - 1)\gamma^2 & = 1,\\
2\beta^2 + 2\kappa^2 + 4\beta\gamma + 4\kappa\lambda + \gamma^2 + \lambda^2 +2\alpha\gamma - n_\cH(n_\cH - 1)(n_\cG - 1)\gamma^2 & = 0
\end{align*}
yields a valid certificate. It is easy to check that $\alpha = n_\cG-1$, $\beta = -1$, $\gamma = 1$, $\kappa = 0$ and $\lambda = 1$ is a solution.

For $n_\cH = 3$ the above equations and also   
\begin{align*}
\beta\gamma + \kappa\lambda + \gamma^2 + \lambda^2 & = 0
\end{align*}
has to be fulfilled in order to obtain a certificate.
Also in this case it can be verified easily that 
$\alpha = \sqrt{2}(n_\cG-1)$, $\beta = -\frac{2}{3}\sqrt{2}$,
 $\gamma = \frac{1}{3}\sqrt{2}$, $\kappa = \frac{1}{3}$ and $ \lambda= -\frac{2}{3}$ is a solution to the system of equations.
Therefore the polynomials $s_0$ and $s_g$ for $g \in V(\cG)$ form indeed a certificate.
\end{proof}

\subsection{Missing Certificates for \texorpdfstring{$k_\cG = n_\cG-1$}{kG=nG-1} and \texorpdfstring{$k_\cH= n_\cH- 1$}{kH=nH-1} with \texorpdfstring{$n_\cH \ge 4$}{nH greater equal 4}}
Previously we have seen that in many cases 
it is possible to obtain a certificate 
not only for particular values of $n_\cH$ and $k_\cH$, 
but for general values, 
like it was done in Section~\ref{sec:proof_of_kG=nG_and_kH=nH-1}. 
Therefore it is a natural question, 
whether we can generalize the certificate 
from Theorem~\ref{thm:sosCertificate_kGEqnGMinus1_kH2nH3} 
for the case $k_\cG = n_\cG - 1 \ge 1$, $n_\cH = 3$ and $k_\cH = 2$
to a certificate for the case $k_\cG = n_\cG-1 \ge 1$ and $k_\cH = n_\cH - 1 \ge 1$.
We have successfully generalized the certificate for $n_\cH= 2$ with 
Theorem~\ref{thm:sosCertificate_kGEqnGMinus1_kH1nH2}.
Unfortunately it turns out that this is not possible for $n_\cH \ge 4$.

\begin{example}
    There seems to be no $2$-sum-of-squares certificate for the case 
    $n_\cG = 4$, $k_\cG=3$, $n_\cH = 4$, $k_\cH = 3$ 
    that uses only the monomials of the form 
    $1$, $x_{gh}$ and $x_{gh}x_{gh'}$ 
    for all $g \in V(\cG)$ and all $\set{h,h'} \subseteq V(\cH)$, 
    as the corresponding SDP is infeasible. 
    
    The SDP for the case $n_\cG = 4$, $k_\cG=3$, $n_\cH = 4$, $k_\cH = 3$ which takes into account all monomials of degree at most $2$ is feasible. Therefore we expect that there is an exact $2$-sum-of-squares certificate using all monomials also for these parameter values.
\end{example}

When we take a closer look on the proofs of 
Section~\ref{sec:auxiliaryResults_kGIsnGMinus1_nH3_kH2} 
and~\ref{sec:certificate_kGEqnGMinus1_nH3_kH2} 
we get some insight in why this is the case. 
First, Lemma~\ref{lem:prodxghj_kGeqnGMinus1_kHeqnHMinusi} 
is not true anymore for $n_\cG \ge 4$, 
so we can not use the reduction 
of all products of 4 variables as presented in 
Remark~\ref{rem:productOf4For_kGisnGMinus1_nH3_kH2}. 
Furthermore $\SigSubset{4}$ is not equal to $0$ anymore for $n_\cG \ge 4$, 
so in the proof of Theorem~\ref{thm:sosCertificate_kGEqnGMinus1_kH2nH3} 
the coefficient of $\SigSubset{4}$ would have to be $0$, 
which is not possible as the coefficient is $6\gamma^2 + 6\lambda^2$.

As a result we would have to search for a certificate 
with more monomials than just $1$, $x_{gh}$ and $x_{gh}x_{gh'}$ 
for the case $k_\cG = n_\cG-1$ and $k_\cH= n_\cH- 1$ for $n_\cH \ge 4$.

\section{Conclusions and Future Work} \label{sec_conc}

\subsection{Conclusions}
In this project, we modeled Vizing's conjecture as an ideal/\allowbreak poly\-nomial
pair such that the polynomial is non-negative on the variety of a
particularly constructed ideal if and only if Vizing's conjecture is
true. We were able to produce low-degree, sparse Positivstellensatz
certificates of non-negativity for certain classes of graphs using an
innovative collection of techniques ranging from semidefinite
programming to clever guesswork to computer algebra. 

In particular, Vizing's conjecture with parameters $k_\cG = n_\cG - 1 \ge 1$, 
$k_\cH = n_\cH - 1$ and $n_\cH \in \set{2,3}$ has a $2$-sum-of-squares Positivstellensatz 
certificate.
Furthermore
Vizing's conjecture with parameters $k_\cG = n_\cG$ and $k_\cH = n_\cH - d$
has a $d$-sum-of-squares Positivstellensatz certificate for~$d\le 4$.
We have conjectured a broader combinatorial
pattern based on these certificates, but proving validity is left to
future work. 

However, at this time, we have indeed proved Vizing's
conjecture for several classes of graphs using sum-of-squares
certificates. Although we have not advanced what is currently known
about Vizing's conjecture, we have introduced a completely new
technique (still to be thoroughly explored) to the literature of
possible approaches.

\subsection{Future Work}
The most pressing matter that arises in this paper is the following.
We have investigated the case  $k_\cG = n_\cG$ and 
$k_\cH = n_\cH - d$. In the future we want to prove 
Conjecture~\ref{con:sosCertificaten_kGeqnG_kHeqnHminusi_easy} or find other 
certificates for the cases $d\ge 5$. In particular it would be interesting to 
know if there is an easy structure for the leading coefficient $\alpha_d$ in 
such a certificate as mentioned in Remark~\ref{remark:question-alpha-d}.

On a small scale, 
in order to obtain more insight on the structure of certificates a
next step will be to investigate further specific parameter
settings. In particular, finding a certificate for the case
$k_\cG = 1$ (and all other parameters arbitrary) is among our next
candidates.

On a large scale, it is known that Vizing's conjecture holds if one of the 
graphs $G$ or
$H$ has domination number at most three~\cite{viz_survey_2009},
therefore, to find new results we need to get certificates for
$k_\cG \ge 4$ and $k_\cH \ge 4$.  Furthermore, it suffices to consider
graphs that contain no isolated vertices. For such graphs the number
of vertices is at least twice the domination number~\cite{ore-1962}. 
Hence, parameters
where we can obtain new results on Vizing's conjecture must satisfy
$n_\cG \ge 8$, $k_\cG \ge 4$, $n_\cH \ge 8$, and $k_\cH \ge 4$.

Therefore, in our
future work we intend to continue pushing the computational aspect of
this project.  One way to do so is to exploit symmetries in order to
simplify the computation of a Gröbner basis, as computing the
Göbner basis is one of the computational bottlenecks.
One alternative possibility to deal with this bottleneck 
is to avoid the computation of a Gröbner basis
by increasing the number of variables in the SDP.
Another question of interest is 
if one could use symmetry 
to reduce the complexity of the SDP.

Up to now we always used the solution of the SDP in order to obtain insight in 
the structure of the certificate and then algebraic manipulations yielded the 
actual certificate. It would be interesting to solve 
the SDP exactly over the algebraic reals and not only to a high precision over the 
rationals, in order to obtain an exact certificate as soon as the SDP is 
solved. However, this is a highly non-trivial task and is left for future 
research.

Another line of research is to change the
model from a Positivstellensatz certificate to a Hilbert's
Nullstellensatz certificate, and thus change from numeric
semidefinite programming to exact arithmetic linear algebra. This
approach must also be thoroughly investigated.

Finally, it would be
very interesting to conjecture a global relationship between the
values of $n_\cG$, $n_\cH$, $k_\cG$ and $k_\cH$, and the degree of the
Positivstellensatz certificate, and perhaps even recast the conjecture
in terms of the theta body hierarchy described in \cite{theta}.

\bibliographystyle{amsplain}
{\small\bibliography{papers}}

{\small
  \vspace*{1ex}\noindent
  Elisabeth Gaar,
  \href{mailto:elisabeth.gaar@aau.at}{\url{elisabeth.gaar@aau.at}},
  Alpen-Adria-Universität Klagenfurt,
  Universitätsstraße 65--67, 9020 Klagenfurt, Austria

  \vspace*{1ex}\noindent
  Daniel Krenn,
  \href{mailto:math@danielkrenn.at}{\url{math@danielkrenn.at}},
  Paris Lodron University of Salzburg,
  Hellbrunnerstraße 34, 5020 Salzburg, Austria
  
  \vspace*{1ex}\noindent
  Susan Margulies,
  \href{mailto:margulie@usna.edu}{\url{margulie@usna.edu}},
  United States Naval Academy, Annapolis, MD, USA
  
  \vspace*{1ex}\noindent
  Angelika Wiegele,
  \href{mailto:angelika.wiegele@aau.at}{\url{angelika.wiegele@aau.at}},
  Alpen-Adria-Universität Klagenfurt,
  Universitätsstraße 65--67, 9020 Klagenfurt, Austria
}

\newpage
\appendix

\section{Appendix: More Complicated ``Intermediate'' Certificates}
\label{sec:firstFoundMoreComplexCertificates}
In 
Sections~\ref{sec:sosCertificaten_kGeqnG_kHeqnH_easy},~\ref{sec:sosCertificaten_kGeqnG_kHeqnHminus1_easy}
 and~\ref{sec:sosCertificaten_kGeqnG_kHeqnHminus2_easy} 
we presented simple sum-of-squares certificates for the case 
$k_\cG = n_\cG$ and $k_\cH = n_\cH-d$ with $d\in \set{0,1,2}$.
In fact, these easy certificates where obtained only after some computational
experiments, in which more complicated certificates were found. We present
these intermediate results and certificates here in this appendix.

For obtaining such a certificate, we use the machinery presented in
Section~\ref{sec:step:numerical-certificate} to get a numerical
certificate. From this, we can guess a structure of the occurring
coefficients, like it was done in Example~\ref{example:3232-2}.  We
will see that these more complicated certificates---they were found by
an SDP solver---have a geometric aspect. By studying this aspect
it was possible to simplify the more complicated certificates to the 
certificates presented in Sections~\ref{sec:sosCertificaten_kGeqnG_kHeqnHminus1_easy} 
and~\ref{sec:sosCertificaten_kGeqnG_kHeqnHminus2_easy}. Hence retrospectively,
these more complicated certificates are formally not needed for the proofs of existence of 
sum-of-squares certificates. Nevertheless we include them here to give a 
more accurate and complete picture of the process of how to obtain certificates.

\subsection{\texorpdfstring{$k_\cG = n_\cG$}{kG=nG} and \texorpdfstring{$k_\cH = n_\cH-1$}{kH=nH-1}}
\label{sec:firstFoundMoreComplexCertificates_d1}
In this case the certificates found by observing a structure and
guessing the coefficients of the numerical certificate have the
following form.

\begin{theorem}\label{thm:sosCertificaten_kGeqnG_kHeqnHminus1}
  For $k_\cG = n_\cG \ge 1$ and $k_\cH = n_\cH-1 \ge 1$, Vizing's conjecture
  is true as the polynomials
    \begin{align*}
    s'_{i} &= \frac{1}{\sqrt{n_\cG}} \sum_{g \in V(\cG)} \lambda_{g,i} \Bigg( \sum_{h \in V(\cH)}  x_{gh}\Bigg)
    \intertext{for $i \in\set{1, \dots, n_\cG-1}$ and}
    s'_{n_\cG} &= \frac{1}{\sqrt{n_\cG}}\Bigg( - k_\cG k_\cH + \sum_{(g,h) \in V(\cG) \times V(\cH)} x_{gh}\Bigg),
    \end{align*}
    where $\lambda_{g,i}$ are solutions to the system of equations
    \begin{subequations}
        \label{systemOfEquations_kEqualsn_kPrimeEqualsnPrimeMinusOne}
        \begin{alignat}{2}
          \sum_{i=1}^{n_\cG-1} \lambda_{g,i}^2 &= n_\cG-1 &
          &\quad\text{for $g \in V(\cG)$},
        \label{systemOfEquations_kEqualsn_kPrimeEqualsnPrimeMinusOne1} \\
          \sum_{i=1}^{n_\cG-1} \lambda_{g,i}\lambda_{g',i} &= -1 &
          &\quad\text{for $\set{g, g'} \subseteq V(\cG)$},
        \label{systemOfEquations_kEqualsn_kPrimeEqualsnPrimeMinusOne2}
        \end{alignat}
    \end{subequations}
    are a $1$-sos certificate of $\fviz$.
\end{theorem}
We can prove this theorem directly, but go a different way here: This
result is one intermediate step and useful and necessary for
conjecturing
Theorem~\ref{thm:sosCertificaten_kGeqnG_kHeqnHminus1_easy}. But once
Theorem~\ref{thm:sosCertificaten_kGeqnG_kHeqnHminus1_easy} is proved,
Theorem~\ref{thm:sosCertificaten_kGeqnG_kHeqnHminus1} is not a
dependency anymore. Therefore, we can reuse the statement made in
Theorem~\ref{thm:sosCertificaten_kGeqnG_kHeqnHminus1_easy} in the
proof here without falling into a cyclic argumentation.

\begin{proof}[Proof of Theorem~\ref{thm:sosCertificaten_kGeqnG_kHeqnHminus1}]
  In Theorem \ref{thm:sosCertificaten_kGeqnG_kHeqnHminus1_easy} we
  have already determined a $1$-sos certificate of $\fviz$ with
  $s_g$, $g\in V(\cG)$, so we know that
  $\sum_{g \in V(\cG)} s_g^2 \equiv \fviz \mod \Iviz$. Hence in
  order to prove that also the $s'_i$ form a certificate, it is enough
  to prove that
  $ \sum_{i=1}^{n_\cG} (s_{i}')^2 = \sum_{g \in V(\cG)} s_g^2$.
  We use the abbreviation $\SigSubset{1} = \sum_{h \in V(\cH)}  x_{gh}$ 
(see Definition~\ref{def:SigSubset_g_i})
  and do this by
  \begin{align*}
    \sum_{i=1}^{n_\cG} (s_{i}')^2
    & = \frac{1}{n_\cG} \Bigg[ \sum_{i=1}^{n_\cG-1} \Bigg(
      \sum_{g \in V(\cG)}\lambda_{g,i}^2\SigSubset{1}^2
      + 2\sum_{\set{g,g'} \subseteq V(\cG)} \lambda_{g,i}\lambda_{g',i}\SigSubset{1}\SigPrimeSubset{1}
    \Bigg) \\
    &\quad \quad \quad \quad
      + (k_\cG k_\cH)^2 - 2k_\cG k_\cH\sum_{g \in V(\cG)}\SigSubset{1}
      + \sum_{g \in V(\cG)}\SigSubset{1}^2
      + 2\sum_{\set{g,g'} \subseteq V(\cG)}\SigSubset{1}\SigPrimeSubset{1}
    \Bigg]\\
    & = \frac{1}{n_\cG} \Bigg[ \sum_{g \in V(\cG)}
      \Bigg(1 +  \sum_{i=1}^{n_\cG-1}\lambda_{g,i}^2 \Bigg) \SigSubset{1}^2
      + 2\sum_{\set{g,g'} \subseteq V(\cG)}
    \Bigg(1 +  \sum_{i=1}^{n_\cG-1}\lambda_{g,i}\lambda_{g',i} \Bigg)  \SigSubset{1}\SigPrimeSubset{1} \\
    & \quad \quad \quad \quad
      + (k_\cG k_\cH)^2 - 2k_\cG k_\cH\sum_{g \in V(\cG)}\SigSubset{1}
    \Bigg] \\
    & \overset{\eqref{systemOfEquations_kEqualsn_kPrimeEqualsnPrimeMinusOne}}{=} \frac{1}{n_\cG} \Bigg[ n_\cG\sum_{g \in V(\cG)} \SigSubset{1}^2 + (k_\cG k_\cH)^2
      - 2k_\cG k_\cH \sum_{g \in V(\cG)}\SigSubset{1}
    \Bigg] \\
    & \overset{k_\cG = n_\cG}{=}
      \sum_{g \in V(\cG)} \SigSubset{1}^2 + k_\cG k_\cH^2 - 2k_\cH\sum_{g \in V(\cG)}\SigSubset{1} \\
    & = \sum_{g \in V(\cG)} \big( \SigSubset{1}^2 - 2k_\cH \SigSubset{1} + k_\cH^2 \big)
    = \sum_{g \in V(\cG)} \big( \SigSubset{1} - k_\cH\big)^2
    = \sum_{g \in V(\cG)} s_g^2,
  \end{align*}
  and so the proof is complete.
\end{proof}

Theorem~\ref{thm:sosCertificaten_kGeqnG_kHeqnHminus1} requires the
solution of a system of equations. We obtain a solution
in the following explicit form.

\begin{lemma}
  Suppose $n_\cG$ is a positive integer and $V(\cG)=\set{1,\dots,n_\cG}$.
  For $g\in V(\cG)$ and $i\in\set{1,\dots,n_\cG-1}$ define
  \begin{equation*}
    \lambda_{g,i} =
    \begin{cases}
      0 &
      \quad\text{for $i<n_\cG-g$}, \\[2ex]
      \displaystyle
      \sqrt{\frac{n_\cG(n_\cG-g)}{n_\cG - g + 1}} &
      \quad\text{for $i=n_\cG-g$}, \\[3ex]
      \displaystyle
      -\frac{\lambda_{n_\cG-i,i}}{i} &
      \quad\text{for $i>n_\cG-g$}.
    \end{cases}
  \end{equation*}
  Then these $\lambda_{g,i}$ are a solution to
  the system of
  equations~\eqref{systemOfEquations_kEqualsn_kPrimeEqualsnPrimeMinusOne}.
\end{lemma}
\begin{proof}
    Consider $\lambda_{g,i}$ to be defined as stated in the lemma. We will show that it satisfies the equations~\eqref{systemOfEquations_kEqualsn_kPrimeEqualsnPrimeMinusOne}.

    We start with an initial remark:
    Observe that $\lambda_{g,i} = \lambda_{g',i}$ whenever 
    $ i > n_\cG-g$ and $i > n_\cG - g'$ hold for all
    $\set{g,g'} \subseteq V(\cG)$.

    First we prove by induction that 
    \begin{align}
    \label{eq:sumPropertyLambdagi}
    \sum_{i=n_\cG - g + 1}^{n_\cG-1} \lambda_{g,i}^2 = 
    \frac{g-1}{n_\cG - g + 1}
    \end{align}
    holds for every $g \in V(\cG)$. Indeed,~\eqref{eq:sumPropertyLambdagi} 
    is trivially satisfied for $g=1$, as both sides are equal to zero.
    In the induction step we assume the~\eqref{eq:sumPropertyLambdagi} holds 
    for $g$
    and prove that it also holds for $g+1$. Towards this end consider
    \begin{align*}
    \sum_{i=n_\cG - g}^{n_\cG-1} \lambda_{g+1,i}^2 
    = \lambda_{g+1,n_\cG - g}^2 + \sum_{i=n_\cG - g + 1}^{n_\cG-1} 
    \lambda_{g+1,i}^2.
    \end{align*}
    Our initial remark implies that
    \begin{align*}
    \sum_{i=n_\cG - g}^{n_\cG-1} \lambda_{g+1,i}^2 
    = \frac{\lambda_{g,n_\cG - g}^2}{(n_\cG-g)^2} + \sum_{i=n_\cG - g + 
    1}^{n_\cG-1} \lambda_{g,i}^2.
    \end{align*}
    By using the induction hypothesis and
    the definition of $\lambda_{g,n_\cG - g}$, this can be further simplified to
    \begin{align*}
    \sum_{i=n_\cG - g}^{n_\cG-1} \lambda_{g+1,i}^2 
    = \frac{1}{(n_\cG-g)^2}\frac{n_\cG(n_\cG-g)}{n_\cG - g + 1} + 
    \frac{g-1}{n_\cG - g + 1}
    = \frac{g}{n_\cG-g}.
    \end{align*}
    Hence, this proves that~\eqref{eq:sumPropertyLambdagi} holds also for $g+1$ 
    and therefore for all $g \in V(\cG)$.
    
    Next we consider the system of equations that has to be satisfied. Observe 
    that $\lambda_{g,i} = 0$ for $i < n_\cG - g$. This 
    and~\eqref{eq:sumPropertyLambdagi} imply
    \begin{align*}
    \sum_{i=1}^{n_\cG-1} \lambda_{g,i}^2
    = \lambda_{g,n_\cG - g}^2 + \sum_{i=n_\cG - g + 1}^{n_\cG-1} 
    \lambda_{g,i}^2
    = \frac{n_\cG(n_\cG-g)}{n_\cG - g + 1} + \frac{g-1}{n_\cG - g + 1}
    = n_\cG -1
    \end{align*}
    and, again because of our initial remark, we have 
    \begin{align*}
    \sum_{i=1}^{n_\cG-1} \lambda_{g,i}\lambda_{g',i}
    &= \lambda_{g,n_\cG - g}\left(-\frac{\lambda_{g,n_\cG-g}}{n_\cG-g} \right) 
        + \sum_{i=n_\cG - g + 1}^{n_\cG-1} \lambda_{g,i}^2\\
    &= -\left(\frac{1}{n_\cG-g}\right)\frac{n_\cG(n_\cG-g)}{n_\cG - g + 1} 
        + \frac{g-1}{n_\cG - g + 1}
    = -1.
    \end{align*}
    Therefore the proposed solution for $\lambda_{g,i}$ is indeed a solution to 
    the system of 
    equations~\eqref{systemOfEquations_kEqualsn_kPrimeEqualsnPrimeMinusOne}.
\end{proof}

\subsection{\texorpdfstring{$k_\cG = n_\cG$}{kG=nG} and \texorpdfstring{$k_\cH = n_\cH-2$}{kH=nH-2}}
\label{sec:firstFoundMoreComplexCertificates_d2}

In this case, we again can find certificates by recognizing a
structure in a numeric certificate and guessing the coefficients. Such
a certificate is of the following form.

\begin{theorem}\label{thm:sosCertificaten_kGeqnG_kHeqnHminus2}
  For $k_\cG = n_\cG \ge 1$ and $k_\cH = n_\cH-2 \ge 1$, Vizing's conjecture
  is true as the polynomials
  \begin{align*}
    s'_{i} &= \frac{1}{\sqrt{n_\cG}}\left( \sum_{g \in V(\cG)} \lambda_{g,i} \left( \sum_{h \in V(\cH)}  x_{gh}\right)
             + \sum_{g \in V(\cG)} \mu_{g,i} \left( \sum_{\set{h, h'} \subseteq V(\cH)}  x_{gh}x_{gh'}\right) \right)
    \intertext{for all $i \in \set{1, \dots, n_\cG-1}$ and}
    s'_{n_\cG} &= \frac{1}{\sqrt{n_\cG}}\left( n_\cG\alpha + \beta \left( \sum_{g \in V(\cG)} \sum_{h \in V(\cH)} x_{gh} \right)
             + \gamma \sum_{g \in V(\cG)} \left( \sum_{\set{h, h'} \subseteq V(\cH)} x_{gh}x_{gh'}\right) \right),
  \end{align*}
  where $\alpha$, $\beta$ and $\gamma$ are solutions of
  \eqref{eq:systemForkHnHMinus2} and $\lambda_{g,i}$ and $\mu_{g,i}$
  are solutions of the system of equations
  \begin{subequations}
    \label{systemOfEquations_kEqualsn_kPrimeEqualsnPrimeMinusTwo}
    \begin{alignat}{2}
    \sum_{i=1}^{n_\cG-1} \lambda_{g,i}^2 &= (n_\cG-1)\beta^2  &&\quad\text{for $g\in V(\cG)$},\\
    \sum_{i=1}^{n_\cG-1} \mu_{g,i}^2 &= (n_\cG-1)\gamma^2  &&\quad\text{for $g \in V(\cG)$},\\
    \sum_{i=1}^{n_\cG-1} \lambda_{g,i}\mu_{g,i} &= (n_\cG-1) \beta\gamma  &&\quad\text{for $g \in V(\cG)$},\\
      \sum_{i=1}^{n_\cG-1} \lambda_{g,i}\lambda_{g',i} &= -\beta^2  &&\quad\text{for $\set{g, g'} \subseteq V(\cG)$},
      \label{systemOfEquations_kEqualsn_kPrimeEqualsnPrimeMinusTwo1}\\
      \sum_{i=1}^{n_\cG-1} \mu_{g,i}\mu_{g',i} &= -\gamma^2  &&\quad\text{for $\set{g, g'} \subseteq V(\cG)$},
      \label{systemOfEquations_kEqualsn_kPrimeEqualsnPrimeMinusTwo2}\\
      \sum_{i=1}^{n_\cG-1} \lambda_{g,i}\mu_{g',i} &= -\beta\gamma  &&\quad\text{for $\set{g, g'} \subseteq V(\cG)$},
      \label{systemOfEquations_kEqualsn_kPrimeEqualsnPrimeMinusTwo3}
    \end{alignat}
  \end{subequations}
  are a $2$-sos certificate of $\fviz$.
\end{theorem}
\begin{proof}[Sketch of the proof]
    The proof can be done analogously to the proof of Theorem~\ref{thm:sosCertificaten_kGeqnG_kHeqnHminus2},
     hence we want to show that 
    \begin{align*}
    \sum_{i=1}^{n_\cG} (s'_{i})^2 = \sum_{g \in V(\cG)} s_{g}^2,
    \end{align*} 
    where the $s_i$ are those from
    Theorem~\ref{thm:sosCertificaten_kGeqnG_kHeqnHminus2_easy}. Towards
    that end we first simplify and express $s'_i$ in terms of
    $\SigSubset{i}$. Then we use binomial expansion to express the
    squares. In the result we can use
    \eqref{systemOfEquations_kEqualsn_kPrimeEqualsnPrimeMinusTwo} in
    order to eliminate all expressions of the form
    $\SigSubset{i}\SigPrimeSubset{j}$ and in order to simplify all
    expressions of the form $\SigSubset{i}\SigSubset{j}$. Eventually
    it is easy to see that the result is in fact a reformulation of
    $\sum_{g \in V(\cG)} s_{g}^2$. Hence the $s'_i$ form a certificate
    due to Theorem~\ref{thm:sosCertificaten_kGeqnG_kHeqnHminus2_easy}.
\end{proof}

\newpage

\end{document}